\documentclass[11pt]{amsart}
\usepackage{amsfonts,amssymb,amsmath,amsthm,amscd}
\usepackage[mathscr]{eucal}
\usepackage{mathrsfs}
\usepackage{color}
\usepackage{enumerate}
\usepackage{here}
\usepackage{longtable}
\usepackage{dcolumn}
\usepackage[dvips]{graphicx}
\usepackage{mathtools}

\theoremstyle{plain}

\newtheorem{theorem}{Theorem}
\newtheorem{corollary}[theorem]{Corollary}
\newtheorem{lemma}[theorem]{Lemma}
\newtheorem{proposition}[theorem]{Proposition}

\theoremstyle{definition}

\newtheorem{example}[theorem]{Example}

\newtheorem{rem}[theorem]{Remark}

\def \Z{{\mathbb Z}}

\def \R{{\mathbb R}}

\def \s{\underset{R}{\sim}^*}

\allowdisplaybreaks[4]
\title[the higher derivatives of the Barnes zeta function]{On values of the higher derivatives of the Barnes zeta function at non-positive integers}
\author{Shinpei Sakane and Miho Aoki}
\date{\today}

\address{Department of Information Systems Design and Data Science,
Major in Science and Engineering,
Graduate School of Natural Science and Technology,
Shimane University,
1060, Nishikawatsu, Matsue, Shimane, 690-8504, Japan }
\email{n21m105@matsu.shimane-u.ac.jp}

\address{Department of Mathematics,
Interdisciplinary Faculty of Science and Engineering,
Shimane University,
1060, Nishikawatsu, Matsue, Shimane, 690-8504, Japan
}
\email{aoki@riko.shimane-u.ac.jp }

\thanks{This work was supported by JSPS KAKENHI Grant Number JP21K03181}

\begin{document}

\maketitle

\begin{abstract}
Let $x$ be a complex number which has a positive real part, and $w_1,\ldots,w_N$ be positive rational numbers.
We show that $w^s \zeta_N (s, x \ |\ w_1,\ldots, w_N)$ can be expressed as a finite linear combination of the
Hurwitz zeta functions over $\mathbb Q(x)$, where $\zeta_N (s,x \ |\ w_1,\ldots, w_N)$ is the
Barnes zeta function and $w$ is a positive rational number explicitly determined by $w_1,\ldots, w_N$.
Furthermore, we give generalizations of Kummer's formula on the gamma function
and Koyama-Kurokawa's formulae on the multiple gamma functions, and an explicit formula
for the values at non-positive integers for higher order derivatives of the Barnes zeta function
in the case that $x$ is a positive rational number,
involving the generalized Stieltjes constants and the values at positive integers of the Riemann zeta function.
Our formulae also makes it possible to calculate an approximation in the case that $w_1, \ldots, w_N$ and $x$ are  positive real numbers.
\end{abstract}

\section{Introduction}\label{sec:Intro}
The Barnes zeta function \cite{B1,B2} is defined by the multiple Dirichlet series:
\[
\zeta_N (s, x\  |\ w_1,\ldots,w_N) := \sum_{n_1=0}^{\infty} \cdots \sum_{n_N=0}^{\infty}
\frac{1}{(x+n_1w_1+\cdots +n_Nw_N)^s}, \hspace{0.073cm} {\rm Re}(s)>N,
\]
where $N$ is a positive integer and $x, w_1, \ldots, w_N$ are fixed complex numbers which have positive
real parts.
The series on the right converges absolutely for ${\rm Re}(s) >N$, and the function of $s$ has a meromorphic continuation to the
entire complex plane with simple poles at $s=1,2,\ldots,N$. For $N=1$, we have
$\zeta_1(s,x \ | \ w)=w^{-s} \zeta(s,x/w)$ where
\[
\zeta (s,x):= \sum_{n=0}^{\infty}(n+x)^{-s}, \quad {\rm Re}(s)>1
\]
is the Hurwitz zeta function.

The Barnes zeta function can be considered as a generalization of the multiple Hurwitz zeta function:
\[
\zeta_N (s,x) :=\zeta_N(s,x \ |\ \underbrace{1,\ldots,1}_{N})
\]
to general periods $w_1,\ldots,w_N$.
Koyama-Kurokawa \cite{KK} showed that $\zeta_N (s,x)$ with $N=2,3$ can be expressed by the Hurwitz zeta
functions :
\begin{align*}
\zeta_2(s,x) &= (1-x)\zeta (s,x)+\zeta (s-1,x), \\
\zeta_3(s,x) &= \frac{1}{2} ( (x^2-3x+2)\zeta (s,x) +(3-2x)\zeta(s-1,x)+\zeta(s-2,x)).
\end{align*}
We first give similar formulae for general $N\in \mathbb Z_{\geq 2}$  in  \S \ref{sec:BH}.

Let
\[
V(x):= \langle  \{ \zeta (s-k,y) \ |\ k\in \mathbb Z_{ \geq 0},  \ y\in \mathbb Q(x),\ {\rm Re}(y) >0 \} \rangle_{\mathbb Q (x)}
\]
be the vector space over the field $\mathbb Q(x)$ generated by the Hurwitz zeta functions.
We next show in \S \ref{sec:BH} that $w^s \zeta_N(s,x \ |\ w_1,\ldots, w_N ) \in V(x)$ for any positive rational numbers
\[
w_1=\frac{r_1}{q_1}, \ \ldots, \ w_N=\frac{r_N}{q_N}
\quad  (r_i,q_i \in \mathbb Z_{>0}, \ {\rm gcd}(r_i,q_i)=1 \ \text{for} \ i=1,\ldots ,N)
\] and $w:={\rm lcm}(r_1,\ldots,r_N)/ {\rm gcd}(q_1,\ldots,q_N)$.

In 1847, Kummer~\cite{K} showed the following idenity:
\begin{align}
  \log \Gamma(x) &= \frac{1}{\pi} \sum_{n=1}^{\infty} \frac{(\log n)\sin(2\pi nx)}{n} + \frac{\log(2\pi)
   + \gamma}{\pi} \sum_{n=1}^{\infty} \frac{\sin(2\pi nx)}{n}  \label{eq:Kummer} \\
  &\quad + \frac{1}{2}\sum_{n=1}^{\infty}\frac{\cos(2\pi nx)}{n} + \frac{1}{2} \log(2\pi) \notag
\end{align}
for $0 < x < 1$.
Koyama and Kurokawa \cite{KK} generalized Kummer's formula on  the gamma function to the multiple gamma function
$\Gamma_N(x):=$exp$(\zeta_N'(0,x))$
with $N=2,3$
by using some formulae for
$\zeta_N'(-k,x)$ for $k\in \mathbb Z_{\geq 1}$:
\begin{align}
  \log \Gamma_2(x) &= -\frac{1}{2\pi^2} \sum_{n=1}^{\infty} \frac{(\log n)\cos(2\pi nx)}{n^2}
   - \frac{\log(2\pi) + \gamma-1}{2\pi^2} \sum_{n=1}^{\infty} \frac{\cos(2\pi nx)}{n^2}  \label{eq:KK1} \\
  &\quad + \frac{1}{4\pi }\sum_{n=1}^{\infty}\frac{\sin(2\pi nx)}{n^2} + (1-x) \log \Gamma_1(x), \notag  \\
   \log \Gamma_3(x) &= -\frac{1}{4\pi^3} \sum_{n=1}^{\infty} \frac{(\log n)\sin(2\pi nx)}{n^3} - \frac{2\log(2\pi)
   + 2\gamma-3}{8\pi^3} \sum_{n=1}^{\infty} \frac{\sin(2\pi nx)}{n^3}  \label{eq:KK2} \\
  &\quad - \frac{1}{8\pi^2 }\sum_{n=1}^{\infty}\frac{\cos(2\pi nx)}{n^3} + \left(\frac{3}{2}-x \right) \log \Gamma_2(x) -\frac{(x-1)^2}{2}\log \Gamma_1(x) , \notag
\end{align}
for $0 < x \leq 1$ (these two equations  hold for $x=1$).
In \S \ref{sec:Kummer}, we  follow their method and give the fomulae for $\Gamma_N(x)$ with $N\in \mathbb Z_{\geq 2}$.
Furthermore, we explain that a similar  formula can be obtained for the multiple gamma function
$\Gamma_N(x|w_1,\ldots,w_N):=$exp$(\zeta_N'(0,x |w_1,\ldots,w_N))$ for positive rational numbers $w_1,\ldots,w_N$ and a positive real numer $x $.

 In \S \ref{sec:val}, we give an explicit formula for
the values at non-positive integers $s$  of higher order derivatives of $\zeta_N(s,x \ |\ w_1,\ldots,w_N)$ for positive
rational numbers $w_1,\ldots,w_N$ and $x$,
which  involves the values at  the positive integers of the Riemann zeta function
  and  the generalized Stieltjes constants.

\section{  Barnes zeta function \and Hurwitz zeta function}\label{sec:BH}

In this section,
we show   $w^s \zeta_N(s,x \ |\ w_1,\ldots, w_N ) \in V(x)$ for any positive rational numbers
\[
w_1=\frac{r_1}{q_1}, \ \ldots, \ w_N=\frac{r_N}{q_N}
\quad  (r_i,q_i \in \mathbb Z_{>0}, \ {\rm gcd}(r_i,q_i)=1 \ \text{for} \ i=1,\ldots ,N)
\] and $w:={\rm lcm}(r_1,\ldots,r_N)/ {\rm gcd}(q_1,\ldots,q_N)$.
\begin{lemma}\label{lem:BZ}
Let $w_1=r_1/q_1,\ldots, w_N =r_N/q_N  \  (r_i,q_i \in \mathbb Z_{ >0}$ for $i=1,2, \ldots, N )$  be  positive  rational numbers
$({\rm gcd}(r_i,q_i)=1$ is not necessary$)$. For any common multiple $q$ of $q_1,\ldots,q_N$ and any common multiple $\ell$ of $qw_1,
\ldots, qw_N \ (\in \mathbb Z)$, we have
\begin{align*}
  &\zeta_N(s,x \ |\ w_1,\ldots, w_N) \\
  &= \left(\frac{q}{\ell}\right)^s \sum_{k_1=0}^{\ell_1-1} \cdots \sum_{k_N=0}^{\ell_N-1}
  \zeta_N \left( s, \frac{q}{\ell} (x+k_1w_1+\cdots+ k_N w_N) \right),
\end{align*}
where $\ell_1:=\ell/(qw_1), \ldots, \ell_N:=\ell/(qw_N) \ (\in \mathbb Z)$.
\end{lemma}
\begin{proof}
For $s\in \mathbb C$ with ${\rm Re}(s) >N$, we have
\begin{align*}
& \zeta_N(s,x \ |\ w_1,\ldots, w_N ) \\
&=\sum_{n_1=0}^{\infty} \cdots \sum_{n_N=0}^{\infty} (x+n_1w_1 +\cdots +n_Nw_N)^{-s} \\
& =\left(\frac{q}{\ell}\right)^s  \sum_{n_1=0}^{\infty} \cdots \sum_{n_N=0}^{\infty} \left( \frac{q}{\ell}
(x+n_1w_1+\cdots +n_Nw_N) \right)^{-s} \\
&= \left(\frac{q}{\ell}\right)^s  \sum_{k_1=0}^{\ell_1-1} \cdots \sum_{k_N=0}^{\ell_N-1}  \sum_{n_1=0  \atop
n_1 \equiv k_1  \!\!\!\! \pmod{\ell_1}}^{\infty} \cdots \sum_{n_N=0 \atop n_N \equiv k_N \!\!\!\! \pmod{\ell_N}}^{\infty} \\
& \hspace{0.4cm}  \times \left\{ \frac{q}{\ell} \left( (x+k_1w_1+\cdots +k_Nw_N) +(n_1-k_1)w_1+\cdots  + (n_N-k_N)w_N\right) \right\}^{-s} \\
& =\left(\frac{q}{\ell}\right)^s  \sum_{k_1=0}^{\ell_1-1} \cdots \sum_{k_N=0}^{\ell_N-1}  \sum_{n_1=0
 }^{\infty} \cdots \sum_{n_N=0}^{\infty} \\
& \hspace{0.5cm} \times \left( \frac{q}{\ell} (x+k_1w_1+\cdots +k_Nw_N)
+n_1+\cdots+n_N \right)^{-s} \\
& =\left(\frac{q}{\ell}\right)^s  \sum_{k_1=0}^{\ell_1-1} \cdots \sum_{k_N=0}^{\ell_N-1}
\zeta_N \left( s,\frac{q}{\ell} (x+k_1w_1+\cdots+k_Nw_N) \right),
\end{align*}
and we get the assertion.
\end{proof}
\begin{proposition}\label{prop:BZ}
Let $w_1=r_1/q_1,\ldots, w_N =r_N/q_N  \  (r_i,q_i \in \mathbb Z_{ >0}$ for $i=1,2, \ldots, N )$  be  positive  rational numbers
with ${\rm gcd}(r_i,q_i)=1$ for $i=1,\ldots,N$  and put  \[
w:= {\rm lcm} (r_1,\ldots,r_N)/{\rm gcd}(q_1,\ldots,q_N) \ (\in \mathbb Q),
\]
$\ell_1:=w/w_1, \ \ldots,\ \ell_N := w/w_N$.
We have $\ell_1, \ldots , \ell_N \ \in \mathbb Z$ and
\begin{align*}
  &\zeta_N(s,x \ |\ w_1,\ldots, w_N) \\
  &= w^{-s} \sum_{k_1=0}^{\ell_1-1} \cdots \sum_{k_N=0}^{\ell_N-1}
  \zeta_N \left( s, w^{-1}(x+k_1w_1+\cdots+ k_N w_N) \right).
\end{align*}
\end{proposition}
\begin{proof}
We put $q:= {\rm lcm} (q_1,\ldots,q_N)$ and $\ell:= {\rm lcm} (qw_1,\ldots, qw_N)$. It is enough to show
$q/\ell = w^{-1}={\rm gcd}(q_1,\ldots,q_N)/{\rm lcm}(r_1,\ldots,r_N)$ from Lemma~\ref{lem:BZ}.
To show the equality, we check ${\rm ord}_p(\ell \  {\rm gcd}(q_1,\ldots,q_N))={\rm ord}_p(q\ {\rm lcm}(r_1,\ldots,r_N))$
 for any prime number $p$, where ${\rm ord}_p(a) \ (\in \mathbb Z)$ for $a\in \mathbb Q^{\times}$ is defined by
 \[
 a=p^{ {\rm ord}_p(a) } \frac{c}{b} ,\quad b,c\in \mathbb Z, \  p\nmid b, \ p\nmid c.
 \]

 Let $p$ be a prime number and $\nu\ (1\leq \nu\leq N)$ be an integer which satisfies
 $
 {\rm ord}_p(w_{\nu} )=\max{  \{ {\rm ord}_p(w_1),\ldots, {\rm ord}_p(w_N) \}  }$.
 Since $\ell ={\rm lcm}(qw_1,\ldots, qw_N)$, we have
 \begin{align}
 {\rm ord}_p(\ell)  & = \max{ \{ {\rm ord}_p(qw_1),\ldots, {\rm ord}_p(qw_N) \} } \notag \\
 & ={\rm ord}_p(q)+{\rm ord}_p(w_{\nu}). \label{eq:maxord}
 \end{align}

 (i) Consider the case $p|q_{\nu}$. From the assumption ${\rm gcd}(r_{\nu},q_{\nu})=1$,
 we have $p\nmid r_{\nu}$, and hence
 \[
 \max{ \{ {\rm ord}_p(w_1),\ldots, {\rm ord}_p(w_N) \} } ={\rm ord}_p(w_{\nu})=-{\rm ord}_p (q_{\nu}) <0.
 \]
 It concludes that ${\rm ord}_p(w_1)<0, \ldots, {\rm ord}_p (w_N)<0$, that is,
 $p \nmid r_1,\ldots , p\nmid r_N$ and $p |q_1,\ldots, p|q_N$.
 From these facts and (\ref{eq:maxord}), we have
 \begin{align*}
 {\rm ord}_p(w_{\nu}) & =\max{ \{ {\rm ord}_p(w_1),\ldots, {\rm ord}_p(w_N) \} } \\
 &=\max{ \{ - {\rm ord}_p(q_1),\ldots, -{\rm ord}_p(q_N) \} } \\
 &=-\min{ \{ {\rm ord}_p(q_1),\ldots, {\rm ord}_p(q_N) \} } \\
 &=- {\rm ord}_p( {\rm gcd}(q_1,\ldots, q_N) ),
 \end{align*}
 and
 \begin{align*}
 {\rm ord}_p (\ell \ {\rm gcd}(q_1,\ldots, q_N))  & ={\rm ord}_p(\ell) +{\rm ord}_p ( {\rm gcd}(q_1,\ldots,q_N)) \\
 & ={\rm ord}_p(q) +{\rm ord}_p(w_{\nu})+{\rm ord}_p({\rm gcd}(q_1,\ldots,q_N)) \\
 & ={\rm ord}_p(q) \\
 & ={\rm ord}_p(q) +{\rm ord}_p( {\rm lcm} (r_1,\ldots,r_N)) \\
 & ={\rm ord}_p (q\ {\rm lcm} (r_1,\ldots , r_N)).
 \end{align*}

 (ii) Consider the case $p\nmid q_{\nu}$. If $p |r_i$ for some $i$,  then $p\nmid q_i$ from ${\rm gcd} (r_i,q_i)=1$ and
 \begin{align*}
 {\rm ord}_p(r_i )  ={\rm ord}_p(w_i) & \leq \max{ \{ {\rm ord}_p(w_1),\ldots {\rm ord}_p(w_N) \} }\\
 & ={\rm ord}_p(w_{\nu}) ={\rm ord}_p (r_{\nu}).
 \end{align*}
 Therefore, we have
 \begin{align*}
   {\rm ord}_p ( {\rm lcm}(r_1,\ldots, r_N)) &= \max{ \{ \rm ord}_p(r_1) , \ldots, {\rm ord}_p(r_N) \} \\
   &={\rm ord}_p(r_{\nu}) ={\rm ord}_p (w_{\nu}).
 \end{align*}
 From these facts and (\ref{eq:maxord}), we have
 \begin{align*}
 {\rm ord}_p( \ell  \ {\rm gcd} (q_1,\ldots,q_N)) & = {\rm ord}_p (\ell) +{\rm ord}_p ({\rm gcd} (q_1,\ldots,q_N))\\
 & ={\rm ord}_p(\ell) \qquad ( \text{since} \ p\nmid q_{\nu}.) \\
 & ={\rm ord}_p(q)+{\rm ord}_p(w_{\nu}) \\
 & ={\rm ord}_p(q) +{\rm ord}_p ({\rm lcm} (r_1,\ldots,r_N)) \\
 & ={\rm ord}_p(q\ {\rm lcm}(r_1, \ldots,r_N)).
 \end{align*}
\end{proof}

For a positive integer $N$, let $C_N(t) \ (\in \mathbb Z[t] )$ be the polynomial
 defined by
 \[
 C_N(t):= \begin{cases}
 1 & (N=1),\\
 (t+N-1)(t+N-2) \cdots (t+1) & (N\geq 2),
 \end{cases}
 \]
 and put
 \[
 C_{N,x}(t):=C_N (t-x) \quad \in (\mathbb Z[x])[t]
 \]
 for $x\in \mathbb C$.
\begin{lemma}\label{lem:zetaC}
We have
\[
\zeta_N (s,x)=\frac{1}{(N-1)!} \sum_{k=0}^{N-1} \frac{ C_{N,x}^{(k)} (0)}{k!} \ \zeta(s-k ,x).
\]
\end{lemma}
\begin{proof}
Let  $s\in \mathbb C$ with ${\rm Re} (s) >N$.  Since
$ | \{ (n_1, \ldots,n_N)  \in \mathbb Z^N \ |\ n_1+\cdots +n_N=n,\ n_1,\ldots, n_N \geq 0 \} | =
\begin{pmatrix}
n+N-1 \\
 N-1\\
 \end{pmatrix}$,
we have
\begin{align*}
\zeta_N(s,x) & =\sum_{n_1=0}^{\infty} \cdots \sum_{n_N=0}^{\infty} (x+n_1+\cdots+n_N)^{-s} \\
& =\sum_{n=0}^{\infty} \begin{pmatrix}
n+N-1 \\
N-1\\
\end{pmatrix} (n+x)^{-s}  \\
& =\frac{1}{(N-1)!} \sum_{n=0}^{\infty} C_{N,x}(n+x) \times (n+x)^{-s} \\
& =\frac{1}{(N-1)!} \sum_{n=0}^{\infty} \sum_{k=0}^{N-1} \frac{C_{N,x}^{(k)} (0) }{k!} (n+x)^{-(s-k)} \\
& = \frac{1}{(N-1)!} \sum_{k=0}^{N-1} \frac{ C_{N,x}^{(k)} (0) }{k!} \ \zeta (s-k,x),
\end{align*}
and we get the assertion.
\end{proof}

From  Proposition~\ref{prop:BZ} and Lemma~\ref{lem:zetaC}, we get the following theorem and corollary.
\begin{theorem}\label{theo:BZ}
Let $x$ be a complex number which has a positive real part, $w_1=r_1/q_1,$
$\ldots, w_N =r_N/q_N  \  (r_i,q_i \in \mathbb Z_{ >0}$ for $i=1,2, \ldots, N )$    positive  rational numbers
with ${\rm gcd}(r_i,q_i)=1$ for $i=1,\ldots,N$ and put
\[
w:= {\rm lcm}(r_1,\ldots,r_N) /{\rm gcd} (q_1,\ldots,q_N) ,\
\ell_1:=w/w_1, \ldots,\ell_N:=w/w_N.
\]
We have
\begin{align*}
& \zeta_N(s,x \ |\ w_1,\ldots,w_N) \\
&=\frac{1}{ w^s (N-1)!} \sum_{k_1=0}^{\ell_1-1} \cdots \sum_{k_N=0}^{\ell_N-1}
\sum_{k=0}^{N-1} \frac{C^{(k)}_{N,y(k_1,\ldots,k_N)}(0)}{k!} \zeta (s-k, y(k_1,\ldots,k_N)).
\end{align*}
where $y(k_1,\ldots,k_N):= w^{-1} (x+k_1w_1+\cdots+k_Nw_N)$.
\end{theorem}
\begin{corollary}\label{cor:BZ2}
Let $ V(x) :=\langle  \{ \zeta (s-k,y) \ |\ k\in \mathbb Z_{ \geq 0},  \ y\in \mathbb Q(x),\ {\rm Re}(y) >0 \} \rangle_{\mathbb Q (x)}$.
For positive rational numbers $w_1,\ldots,w_N$, we have
$w^s \zeta_N (s, x\ |\ $ $w_1$, $\ldots, w_N) \ \in V(x)$.
\end{corollary}
\newpage
\begin{example}\label{ex:BZ}
{\small
(1)
\vspace{-1cm}
\begin{align*}
\zeta_1(s,x) &= \zeta(s,x), \\
\zeta_2(s,x) &= (1-x)\zeta (s,x)+\zeta (s-1,x), \\
\zeta_3(s,x) &= \frac{1}{2} ( (x^2-3x+2)\zeta (s,x) +(3-2x)\zeta(s-1,x)+\zeta(s-2,x)).
\end{align*}
(2)
\vspace{-1cm}
\begin{align*}
 &\zeta_1(s,x \ |\ 1) = \zeta(s,x), \\
 &\zeta_{2}(s, x \mid 1, 2)
= 2^{-s}\left(\left(1 - \frac{x}{2}\right)\zeta\left(s, \frac{x}{2}\right) + \zeta\left(s - 1, \frac{x}{2}\right) + \left(\frac{1}{2} - \frac{x}{2}\right)\zeta\left(s, \frac{x}{2} + \frac{1}{2}
\right) + \zeta\left(s - 1, \frac{x}{2} + \frac{1}{2}\right)\right)
\end{align*}
(3)
\vspace{-1cm}
\begin{align*}
 &\zeta_1(s,x \ |\ 1) = \zeta(s,x),  \hspace{5cm} \\
  &\zeta_{2}\left(s, x \middle|\, 1, \frac{1}{2}\right)
  = \left(1 - x\right)\zeta\left(s, x\right) + \zeta\left(s - 1, x\right) + \left(\frac{1}{2} - x\right)\zeta\left(s, x + \frac{1}{2}\right) + \zeta\left(s - 1, x + \frac{1}{2}\right), \\
 &\zeta_{3}\left(s, x \middle|\, 1, \frac{1}{2}, \frac{1}{3}\right) \\
 &=\frac{1}{2}\left(\left(x^{2} - 3 x + 2\right)\zeta\left(s, x\right) + \left(3 - 2 x\right)\zeta\left(s-1, x\right) + \zeta\left(s-2, x\right)\right. \\
 & \left. +\left(x^{2} - \frac{7 x}{3} + \frac{10}{9}\right)\zeta\left(s, x + \frac{1}{3}\right) + \left(\frac{7}{3} - 2 x\right)\zeta\left(s-1, x + \frac{1}{3}\right) + \zeta\left(s-2, x + \frac{1}{3}\right)\right. \\
 & \left. +\left(x^{2} - \frac{5 x}{3} + \frac{4}{9}\right)\zeta\left(s, x + \frac{2}{3}\right) + \left(\frac{5}{3} - 2 x\right)\zeta\left(s-1, x + \frac{2}{3}\right) + \zeta\left(s-2, x + \frac{2}{3}\right)\right. \\
 & \left. +\left(x^{2} - 2 x + \frac{3}{4}\right)\zeta\left(s, x + \frac{1}{2}\right) + \left(2 - 2 x\right)\zeta\left(s-1, x + \frac{1}{2}\right) + \zeta\left(s-2, x + \frac{1}{2}\right)\right. \\
 & \left. +\left(x^{2} - \frac{4 x}{3} + \frac{7}{36}\right)\zeta\left(s, x + \frac{5}{6}\right) + \left(\frac{4}{3} - 2 x\right)\zeta\left(s-1, x + \frac{5}{6}\right) + \zeta\left(s-2, x + \frac{5}{6}\right)\right. \\
 & \left. +\left(x^{2} - \frac{2 x}{3} - \frac{5}{36}\right)\zeta\left(s, x + \frac{7}{6}\right) + \left(\frac{2}{3} - 2 x\right)\zeta\left(s-1, x + \frac{7}{6}\right) + \zeta\left(s-2, x + \frac{7}{6}\right)\right).
\end{align*}
}
\end{example}
The surface on the top in Figure~\ref{fig:bzeta1-1/2} shows  a graph of $\zeta_2(s,x  \ |\ 1,1/2)$ with $-2 <s<1$ and $0<x<1$,
and the curve on the  bottom shows the section of the surface cut by the plane $x=1/3$.
From the figure, we can know that $\zeta_2 (s,1/3  \ | \ 1,1/2)$ has a zero around $s=0.25$. In fact,
the computation shows that the zero exists at $s \fallingdotseq 0.2558028917231215$.
\begin{figure}[H]
\begin{center}
\includegraphics[width=10cm]{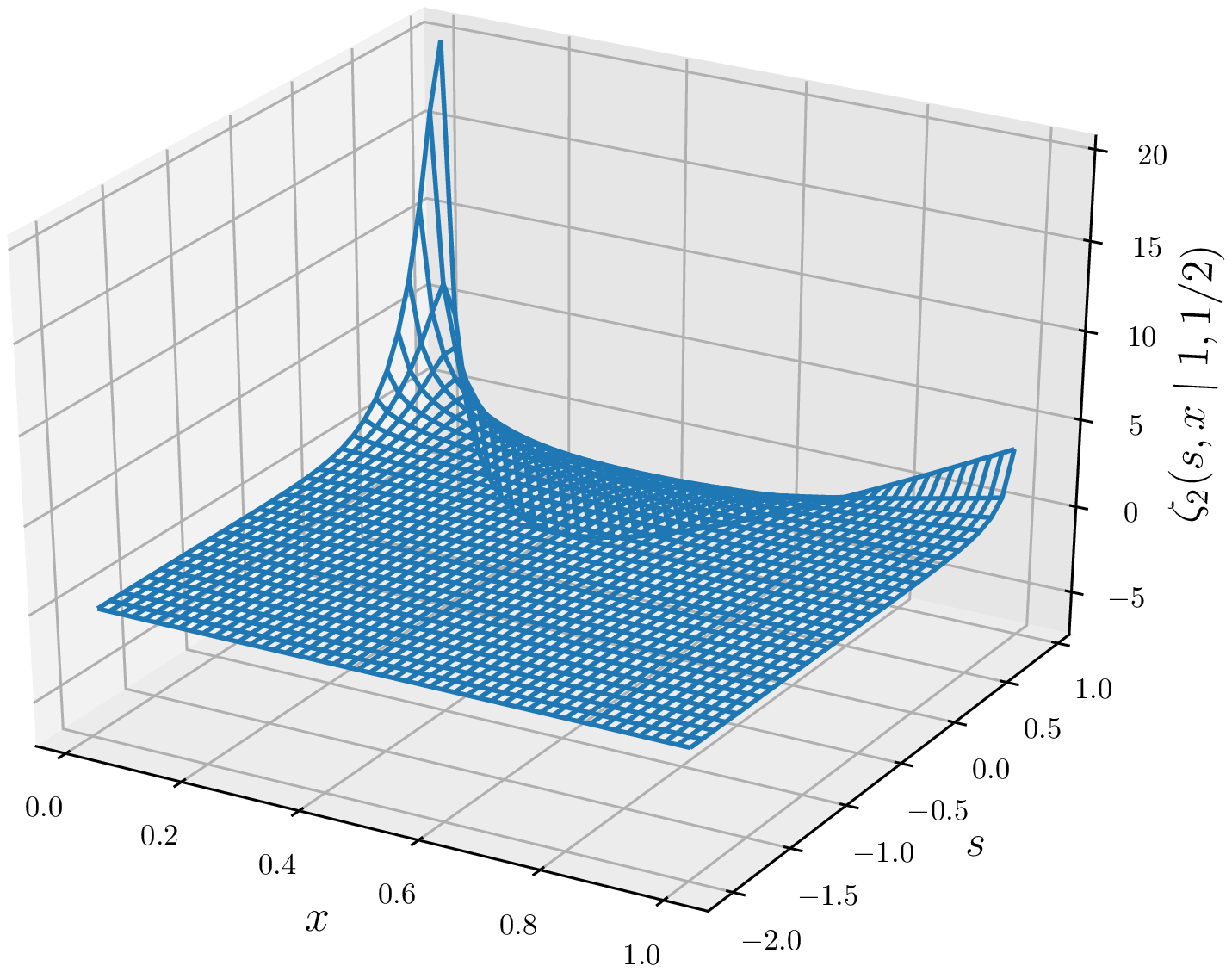}
\includegraphics[width=10cm]{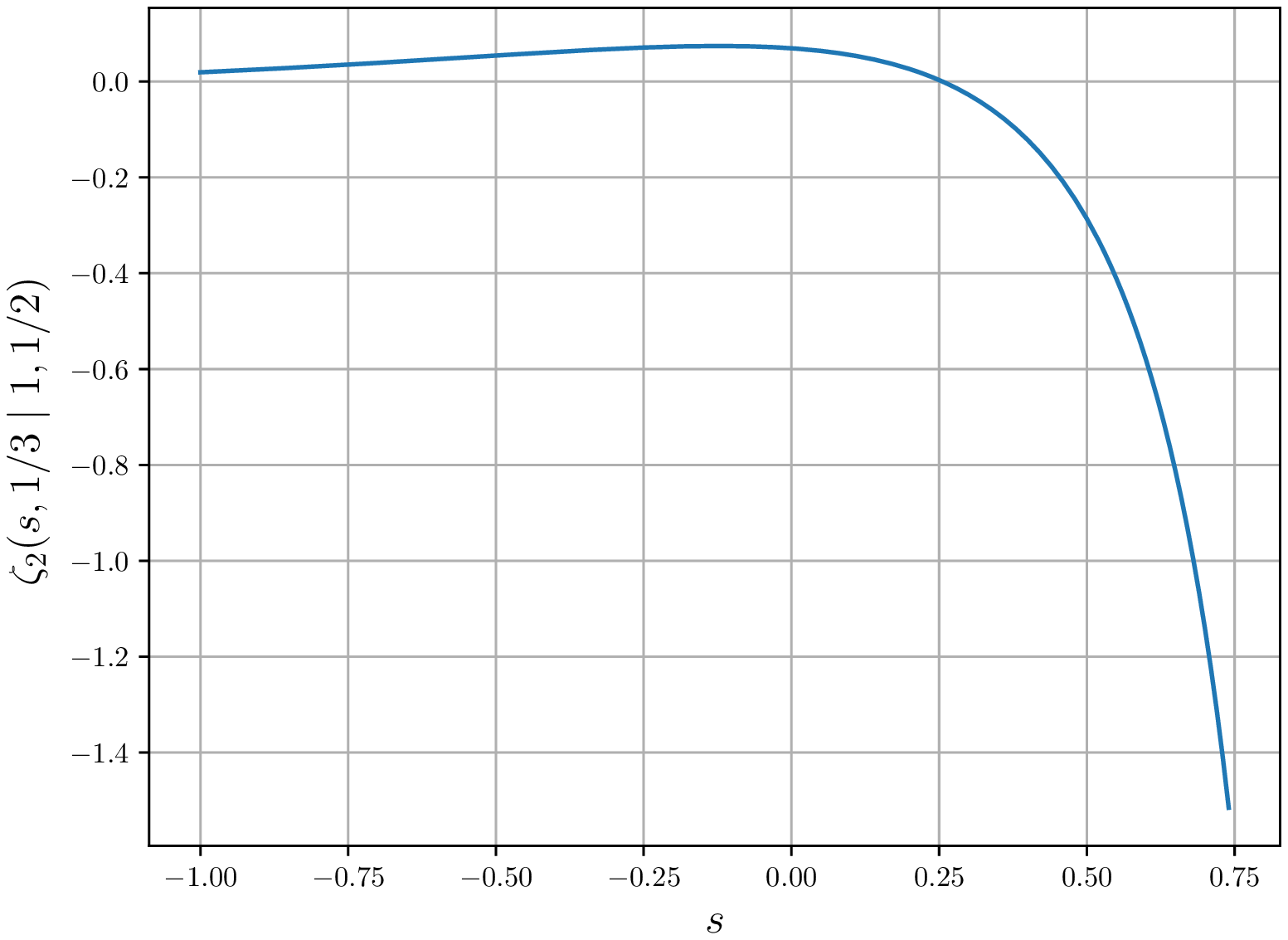}
\end{center}
\caption{$\zeta_2(s,x\ |\ 1,1/2)$ and $\zeta_2(s,1/3 \ |\ 1,1/2)$}
\label{fig:bzeta1-1/2}
\end{figure}
The surface on the top in Figure~\ref{fig:bzeta1-1/2-1/3} shows  a graph of $\zeta_3(s,x  \ |\ 1,1/2,1/3)$ with $-2 <s<1$ and $0<x<1$,
and the curve on the  bottom shows the section of the surface cut by the plane $x=1/3$.
From the figure, we can know that $\zeta_3 (s,1/3  \ | \ 1,1/2,1/3)$ has a zero around $s=0$. In fact,
we have $\zeta_3(0,1/3 \ |\ 1/2,1/3)=0$.
We can  know this fact by a result of the Barnes on the relation between the values of Barnes zeta function at non-positive integers
and the Bernoulli-Barnes polynomials $B_k (X \ |\ w_1,\ldots,w_N) \in \mathbb Q(w_1,\ldots,w_N)[X]$ (\cite[p.383]{B2}) defined by
\[
\frac{t^N e^{(w_1+\cdots+w_N-X)t}}{(e^{w_1t}-1) \cdots (e^{w_Nt}-1)} =\sum_{k=0}^{\infty} B_k(X \ |\ w_1,\ldots,w_N) \frac{t^k}{k!}.
\]
\begin{figure}[H]
\begin{center}
\includegraphics[width=10cm]{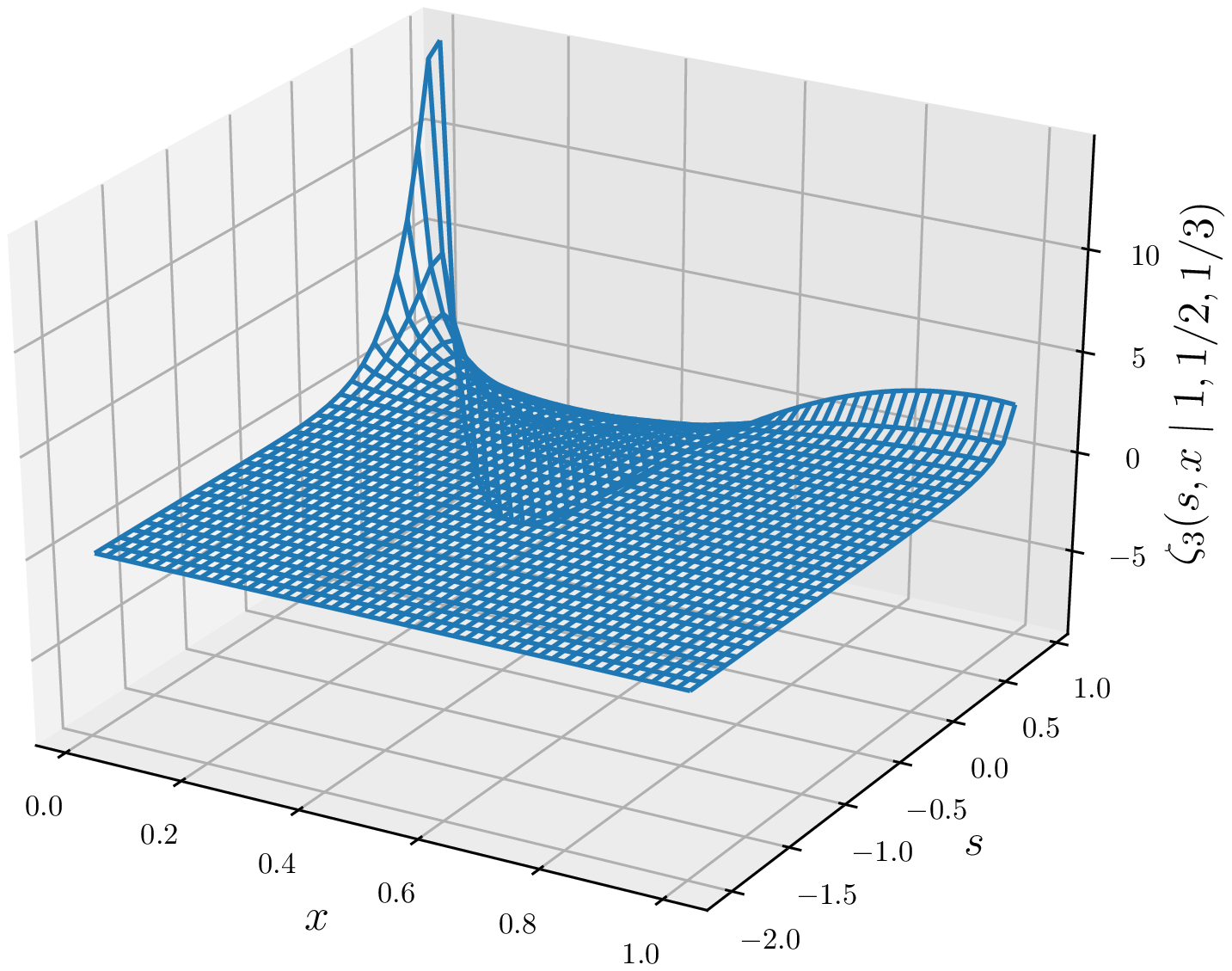}
\includegraphics[width=10cm]{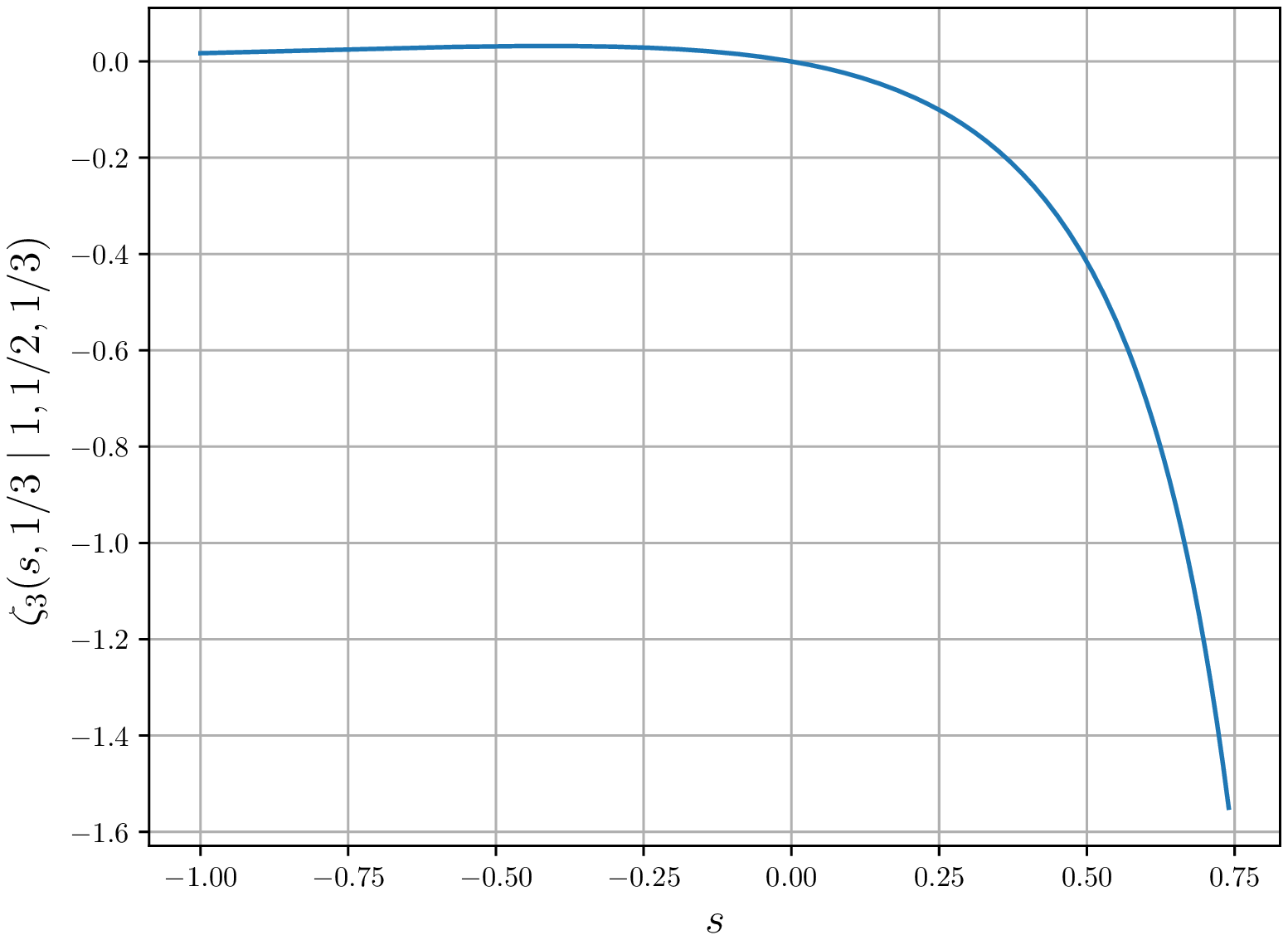}
\end{center}
\caption{$\zeta_3(s,x\ |\ 1,1/2,1/3)$ and $\zeta_3(s,1/3 \ |\ 1,1/2,1/3)$}
\label{fig:bzeta1-1/2-1/3}
\end{figure}
\begin{theorem}\label{theo:Barnes} $($Barnes \cite[p.405]{B2}$)$
For $\ell \in \mathbb Z_{\geq 0}$, we have
\[
\zeta_N(-\ell, x \ |\ w_1,\ldots, w_N)=\frac{(-1)^{\ell}\ell!}{(N+\ell)!} B_{N+\ell} (x \ |\ w_1,\ldots,w_N).
\]
\end{theorem}
By Theorem~\ref{theo:Barnes} and
\begin{align*}
B_3(X \ | \ w_1,w_2,w_3) &=  \frac{3}{4}+\frac{w_1+w_2}{4w_3}+\frac{w_2+w_3}{4w_1}+\frac{w_3+w_1}{4w_2} \\
& - \left\{ \frac{3}{2 w_1}+\frac{3}{2w_2}+\frac{3}{2w_3}+\frac{w_1}{2w_2w_3}+\frac{w_2}{2w_3w_1}+\frac{w_3}{2w_1w_2} \right\} X \\
& +\left\{ \frac{3}{2w_1w_2}+\frac{3}{2w_2w_3}+\frac{3}{2w_1w_3} \right\} X^2 -\frac{1}{w_1w_2w_3} X^3,
\end{align*}
we have
\[
\zeta_3 (0,1/3 \ | \ 1,1/2,1/3)=\frac{1}{3!} B_3(1/3 \ | \ 1,1/2,1/3)=0.
\]

\section{Generalization of Kummer's Formula}\label{sec:Kummer}

In this section, we give generalizations of Kummer's formula (\ref{eq:Kummer}).
We can consider Kummer's formula for the multiple gamma functions:
\begin{align*}
\Gamma_N (x) & :=\exp(\zeta_N'(0,x)), \\
\Gamma_{N}(x|w_{1}, \ldots, w_{N}) &:= \exp\left( \zeta^{\prime}_{N}(0, x| w_{1}, \ldots, w_{N}) \right)
\end{align*}
defined by Barnes~\cite{B2}.
By using $\log \Gamma_{1}(x) = \zeta^{\prime}(0, x) = \log (\Gamma(x)/\sqrt{2\pi})$ and Kummer's formula, we have
\begin{align*}
  \log \Gamma_{1}(x) &= \frac{1}{\pi} \sum_{n=1}^{\infty} \frac{(\log n)\sin(2\pi nx)}{n} + \frac{\log(2\pi) + \gamma}{\pi} \sum_{n=1}^{\infty} \frac{\sin(2\pi nx)}{n} \\
  &\quad + \frac{1}{2}\sum_{n=1}^{\infty}\frac{\cos(2\pi nx)}{n}.
\end{align*}
Koyama and Kurokawa~\cite[Theorem~1]{KK} generalized Kummer's formula for $\Gamma_{N}(x)$ with $N = 2, 3$.
In this section, we generalize Kummer's formula for $\Gamma_{N}(x)$ with $N \in \Z_{\ge 2}$ and explain that a similar formula can be obtained for
$\Gamma_N (x|w_1,\ldots,w_N)$ for
positive rational numbers $w_{1}, \ldots, w_{N}$
and a positive real number $x$.

First, we show Lemmas~\ref{lem:zetaNandN-1} and \ref{lem:zetaN-k}.

\begin{lemma}\label{lem:zetaNandN-1}
  For $N \in \Z_{\ge 2}$, we have
  \[(N-1) \zeta_{N}(s, x) = \zeta_{N-1}(s-1, x) + (N-1-x) \zeta_{N-1}(s, x).\]
\end{lemma}

\begin{proof}
  Using $C_{N, x}(t) = (t - x + N -1)C_{N-1, x}(t)$, we can prove the following formula by induction on $k$:
  \begin{equation}\label{eq:CNandN-1}
    C_{N, x}^{(k)}(t) = kC_{N-1, x}^{(k-1)}(t) + (t - x + N -1)C_{N-1, x}^{(k)}(t).
  \end{equation}

  From Lemma~\ref{lem:zetaC} and \eqref{eq:CNandN-1}, we have
  \begin{align*}
    &\zeta_{N-1}(s-1, x) + (N-1-x)\zeta_{N-1}(s, x) \\
    &= \frac{1}{(N-2)!} \sum_{k=0}^{N-2} \frac{C_{N-1, x}^{(k)}(0)}{k!} \zeta(s-k-1, x) + \frac{N-1-x}{(N-2)!} \sum_{k=0}^{N-2} \frac{C_{N-1, x}^{(k)}(0)}{k!} \zeta(s-k, x) \\
    &= \frac{1}{(N-2)!} \sum_{k=0}^{N-1} \frac{kC_{N-1, x}^{(k-1)}(0)}{k!} \zeta(s-k, x) + \frac{N-1-x}{(N-2)!} \sum_{k=0}^{N-1} \frac{C_{N-1, x}^{(k)}(0)}{k!} \zeta(s-k, x) \\
    &= \frac{1}{(N-2)!} \sum_{k=0}^{N-1} \frac{kC_{N-1, x}^{(k-1)}(0) + (N-1-x)C_{N-1, x}^{(k)}(0)}{k!} \zeta(s-k, x) \\
    &= \frac{1}{(N-2)!} \sum_{k=0}^{N-1} \frac{C_{N, x}^{(k)}(0)}{k!} \zeta(s-k, x) \\
    &= (N-1) \zeta_{N}(s, x).
  \end{align*}
\end{proof}

\begin{lemma}\label{lem:zetaN-k}
  For $k \in \Z_{\ge 1}$, we have
  \[\zeta_{N-k}(s-k, x) = \sum_{\ell = 0}^{k} (-1)^{\ell} \zeta_{N-\ell}(s, x) \prod_{i=1}^{k-\ell} (N-k+i-1) \sum_{N-k \le a_{1} \le \cdots \le a_{\ell} \le N-\ell} (a_{1} - x) \cdots (a_{\ell} - x),\]
  where
  \[\sum_{N-k \le a_{1} \le \cdots \le a_{\ell} \le N-\ell} (a_{1} - x) \cdots (a_{\ell} - x) = 1\]
  when $\ell = 0$, and
  \[\prod_{i=1}^{k-\ell} (N-k+i-1) = 1\]
  when $\ell = k$.
\end{lemma}

\begin{proof}
  We prove this lemma by induction on $k$.
  By Lemma~\ref{lem:zetaNandN-1}, it holds for $k = 1$.

  Next, we assume that the assertion holds for $k \in \Z_{\ge 1}$.
  Then, using Lemma~\ref{lem:zetaNandN-1}, we have
  \begin{align*}
    &\zeta_{N-(k+1)}(s-(k+1), x) = (N-k-1)\zeta_{N-k}(s-k, x) - (N-k-1-x)\zeta_{N-k-1}(s-k, x) \\
    &= \sum_{\ell = 0}^{k} (-1)^{\ell} \zeta_{N-\ell}(s, x) \prod_{i=0}^{k-\ell} (N-k+i-1) \sum_{N-k \le a_{1} \le \cdots \le a_{\ell} \le N-\ell} (a_{1} - x) \cdots (a_{\ell} - x) \\
    &\quad - \sum_{\ell = 0}^{k} (-1)^{\ell} \zeta_{N-\ell-1}(s, x) \prod_{i=1}^{k-\ell} (N-k+i-2) \\
    &\qquad \times \sum_{N-k-1 = a_{1} \le \cdots \le a_{\ell+1} \le N-\ell-1} (a_{1} - x) \cdots (a_{\ell+1} - x) \\
    &= \sum_{\ell = 0}^{k} (-1)^{\ell} \zeta_{N-\ell}(s, x) \prod_{i=1}^{k-\ell+1} (N-k+i-2) \sum_{N-k \le a_{1} \le \cdots \le a_{\ell} \le N-\ell} (a_{1} - x) \cdots (a_{\ell} - x) \\
    &\quad - \sum_{\ell = 1}^{k+1} (-1)^{\ell-1} \zeta_{N-\ell}(s, x) \prod_{i=1}^{k-\ell+1} (N-k+i-2) \sum_{N-k-1 = a_{1} \le \cdots \le a_{\ell} \le N-\ell} (a_{1} - x) \cdots (a_{\ell} - x) \\
    &= \prod_{i=1}^{k+1} (N-k+i-2) \zeta_{N}(s, x) \\
    &\quad + \sum_{\ell = 1}^{k} (-1)^{\ell} \zeta_{N-\ell}(s, x) \prod_{i=1}^{k-\ell+1} (N-k+i-2) \sum_{N-k-1 \le a_{1} \le \cdots \le a_{\ell} \le N-\ell} (a_{1} - x) \cdots (a_{\ell} - x) \\
    &\quad + (-1)^{k+1} (N-k-1-x)^{k+1} \zeta_{N-k-1}(s, x) \\
    &= \sum_{\ell = 0}^{k+1} (-1)^{\ell} \zeta_{N-\ell}(s, x) \prod_{i=1}^{k-\ell+1} (N-k+i-2) \sum_{N-k-1 \le a_{1} \le \cdots \le a_{\ell} \le N-\ell} (a_{1} - x) \cdots (a_{\ell} - x).
  \end{align*}
  Thus, the assertion also holds for $k+1$.
\end{proof}

Using Lemma~\ref{lem:zetaN-k} for $k = N-1$, we get the following corollary.

\begin{corollary}\label{cor:zetaN-1Gamma}
  For $N \in \Z_{\ge 1}$, we have
  \[\zeta^{\prime}(-(N-1), x) = \sum_{\ell = 0}^{N-1} (-1)^{\ell} \zeta^{\prime}_{N-\ell}(0, x) \prod_{i=1}^{N-\ell-1} i \sum_{1 \le a_{1} \le \cdots \le a_{\ell} \le N-\ell} (a_{1} - x) \cdots (a_{\ell} - x).\]
\end{corollary}

From Corollary~\ref{cor:zetaN-1Gamma} and \cite[Theorem~3]{KK}, we have the following theorem.

\begin{theorem}\label{th:logG}
  For $x \in \R$ with $0 < x \le 1$ and $N \in \Z_{\ge 2}$, we have
  \begin{align*}
    &\sum_{\ell = 0}^{N-1} (-1)^{\ell} \log \Gamma_{N-\ell}(x) \prod_{i=1}^{N-\ell-1} i \sum_{1 \le a_{1} \le \cdots \le a_{\ell} \le N-\ell} (a_{1} - x) \cdots (a_{\ell} - x) \\
    &=
    \begin{dcases}
      \begin{aligned}
        &\frac{2(-1)^{\frac{N}{2}}(N-1)!}{(2\pi)^{N}} \left\{ \sum_{n=1}^{\infty} \frac{(\log n) \cos(2\pi nx)}{n^{N}} \right. \\
        &\left. + \left( \log(2\pi) + \gamma - \left(1 + \frac{1}{2} + \cdots + \frac{1}{N-1} \right) \right) \sum_{n=1}^{\infty} \frac{\cos(2\pi nx)}{n^{N}} \right. \\
        &\left. \hspace{6.6cm} - \frac{\pi}{2} \sum_{n=1}^{\infty} \frac{\sin(2\pi nx)}{n^{N}} \right\}
      \end{aligned} & (N \in 2\Z), \\
      \begin{aligned}
        &\frac{2(-1)^{\frac{N-1}{2}}(N-1)!}{(2\pi)^{N}} \left\{ \sum_{n=1}^{\infty} \frac{(\log n) \sin(2\pi nx)}{n^{N}} \right. \\
        &\left. + \left( \log(2\pi) + \gamma - \left(1 + \frac{1}{2} + \cdots + \frac{1}{N-1} \right) \right) \sum_{n=1}^{\infty} \frac{\sin(2\pi nx)}{n^{N}} \right. \\
        &\left. \hspace{6.6cm} + \frac{\pi}{2} \sum_{n=1}^{\infty} \frac{\cos(2\pi nx)}{n^{N}} \right\}
      \end{aligned} & (N \not \in 2\Z). \\
    \end{dcases}
  \end{align*}
\end{theorem}

\begin{rem}
  \cite[Theorems~1 and 3]{KK} hold for $x = 1$.
  This is confirmed by \eqref{eq:Leib} and
  \[\Gamma^{\prime}(1+k) = k!\left(- \gamma + \sum_{\ell = 1}^{k} \frac{1}{\ell} \right)\]
  for $k \in \Z_{\ge 1}$.
  In the case of $N=2,3$ in Theorem~\ref{th:logG}, we get the formulae (\ref{eq:KK1}) and (\ref{eq:KK2}).
\end{rem}

Next, we explain that a similar formula can be obtained for
$\Gamma_N(x|w_1,\ldots,w_N)$  for
positive rational numbers $w_{1}, \ldots, w_{N}$ and $x \in \R$ with $x > 0$.
By Theorem~\ref{theo:BZ}, we have
\begin{align*}
  & \log \Gamma_{N}(x; w_{1}, \ldots, w_{N}) = \zeta_N^{\prime}(0,x \ |\ w_1,\ldots,w_N) \\
  &= \frac{1}{(N-1)!} \sum_{k_1=0}^{\ell_1-1} \cdots \sum_{k_N=0}^{\ell_N-1}
  \sum_{k=0}^{N-1} \frac{C^{(k)}_{N,y(k_1,\ldots,k_N)}(0)}{k!} \\
  &\qquad \times \left\{ \zeta^{\prime} (-k, y(k_1,\ldots,k_N)) - (\log w) \zeta (-k, y(k_1, \ldots ,k_N) \right\},
\end{align*}
where $w \in \Z_{\ge 1}$ and $y(k_{1}, \ldots, k_{N}) \in \R_{> 0}$ given in Theorem~\ref{theo:BZ}.
Furthermore, for $y \in \R_{> 0}$, we have
\[\zeta(s, y) = \zeta(s, x) - \sum_{m=0}^{Y-1} (m + x)^{-s},\]
where $Y \in \Z_{\ge 0}$ and $x \in \R$ with $0 < x \le 1$ satisfy $y = Y + x$.
Therefore, using Corollary~\ref{cor:zetaN-1Gamma}, we can represent $\zeta^{\prime}(-k, y(k_{1}, \ldots, k_{N}))$ by $\log \Gamma_{N-\ell}(x)$ for $\ell = 0, 1, \ldots, k$.

\section{ Values at non-positive integers  of higher order derivatives   }\label{sec:val}

In this section,
we consider
the values at non-positive integers $s$  of higher order derivatives of $\zeta_N(s,x \ |\ w_1,\ldots,w_N)$ for positive
rational numbers $w_1,\ldots,w_N$ and $x$. From Theorem~\ref{theo:BZ}, we have
\begin{align}
 &\zeta_N^{(n)} (s,x \ |\ w_1,\ldots,w_N) =\frac{1}{(N-1)!} \sum_{j=0}^n \begin{pmatrix}
n\\
j\\
\end{pmatrix} (-\log{w})^{n-j}  w^{-s}  \notag \\
& \hspace{2cm} \times \sum_{k_1=0}^{\ell_1-1} \cdots \sum_{k_N=0}^{\ell_N-1} \sum_{k=0}^{N-1}  \frac{C^{(k)}_{N,y(k_1,\ldots,k_N)}(0)}{k!} \zeta^{(j)} (s-k, y(k_1,\ldots,k_N)),
\label{eq:higher-zeta0}
\end{align}
for $n\in \mathbb Z_{\geq 0}$. Therefore, the calculation of the values at   non-positive integers $s$ of the higher order derivatives of $\zeta_N(s,x \ |\ w_1,\ldots, w_N)$ for
positive rational numbers $w_1,\ldots, w_N$ and $x$ is reduced to that of the values of $\zeta^{(n)}(s,y)$  at non-positive integers $s $ for
$y\in \mathbb Q_{>0}$.

Let $u$ and $v$ be integers with $1\leq u\leq v$. Consider the functional equation \cite[Theorem~12.8]{A1}:
\begin{align}
\zeta \left(1-s,\frac{u}{v} \right) & =  \frac{2\Gamma (s) } {(2\pi v)^s} H\left( s,\frac{u}{v} \right), \label{eq:Feq}
\end{align}
where
\[
H \left( s,\frac{u}{v} \right):= \sum_{k=1}^v \left\{ \cos{ \left( \frac{\pi s}{2} -\frac{2\pi ku}{v} \right) } \zeta \left(
s,\frac{k}{v} \right) \right\}.
\]

We have
\begin{equation}\label{eq:Tayzeta}
\zeta (s,x)=\frac{1}{s-1}+\sum_{n=0}^{\infty} \frac{(-1)^n}{n!} \gamma_n(x) (s-1)^n,
\end{equation}
where $\gamma_n(x)$ is the generalized Stieltjes constant.
For the Riemann zeta function $\zeta (s)=\zeta (s,1)$,
\[
 \gamma_n=\gamma_n(1)=\lim_{m\to \infty} \left( \sum_{k=1}^m \frac{ \log^n{k}}{k}-\frac{\log^{n+1}{m}}{n+1} \right),
\]
and
\[
\gamma:=\gamma_0= \lim_{m\to \infty} \left( \sum_{k=1}^m \frac{ 1}{k}-\log{m} \right) \fallingdotseq 0.57721
\]
is the Euler's constant.
Berndt \cite{B} proved the following theorem.
\begin{theorem}[\protect{\cite[Theorem 1]{B}}]\label{theo:Berndt}
For $x\in \mathbb R$ with $0<x\leq 1$ and $n\in \mathbb Z$ with $n\geq 0$, we have
\[
\gamma_n(x)=\lim_{m\to \infty} \left(\sum_{k=0}^m \frac{\log^n(k+x)}{k+x}-\frac{\log^{n+1}(m+x)}{n+1} \right).
\]
\end{theorem}

\begin{lemma}\label{lem:Entire}
$H\left(s,\frac{u}{v} \right)$ is an entire function.
\end{lemma}
\begin{proof}
It is enough to show that $H\left( s,\frac{u}{v} \right)$ is holomorphic at $s=1$. Since
\[
\cos{\left( \frac{\pi s}{2} -\frac{2\pi ku}{v} \right)}^{(n)}=-\left(\frac{\pi}{2} \right)^n \sin{ \left( \frac{\pi (s+n-1)}{2} -
\frac{2\pi ku}{v} \right)},
\]
we have
\begin{equation}\label{eq:cos=1}
\cos{\left( \frac{\pi s}{2} -\frac{2\pi ku}{v} \right)}^{(n)} \vline_{\begin{array}{l} \\ s=1 \end{array} }=-\left(\frac{\pi}{2} \right)^n \sin{ \left( \frac{\pi n}{2} -
\frac{2\pi ku}{v} \right)}.
\end{equation}
Therefore, we obtain the Taylor expansion at $s=1$:
\begin{equation}\label{eq:Tayc}
\cos{\left( \frac{\pi s}{2} -\frac{2\pi ku}{v} \right)} =-\sum_{n=0}^{\infty} \frac{1}{n!} \left( \frac{\pi}{2} \right)^n \sin{ \left( \frac{\pi n}{2}-\frac{2\pi ku}{v} \right)} (s-1)^n.
\end{equation}

From (\ref{eq:Tayc}) and (\ref{eq:Tayzeta}),
the coefficient of $(s-1)^{-1}$ in the expansion of $H\left(s,\frac{u}{v} \right)$ is $\displaystyle{ \sum_{k=1}^v {\rm sin} \left( \frac{2\pi ku}{v} \right)=0}$, and hence
$H\left(s,\frac{u}{v} \right)$ is holomorphic at $s=1$ and
\begin{align}
H\left(s,\frac{u}{v} \right) & = -\sum_{n=1}^{\infty} \sum_{k=1}^v \frac{1}{n!} \left( \frac{\pi}{2} \right)^n {\rm sin} \left( \frac{n\pi}{2} -\frac{
2\pi ku}{v} \right) (s-1)^{n-1} \notag \\
&+\sum_{k=1}^v \left\{ \cos{ \left( \frac{\pi s}{2} -\frac{2\pi ku}{v} \right) } \sum_{n=0}^{\infty} \frac{(-1)^n} {n!} \gamma_n \left(
\frac{k}{v} \right) (s-1)^n \right\}. \label{eq:H}
\end{align}
\end{proof}

By Leibniz's rule for (\ref{eq:Feq}), we have
\begin{eqnarray}\label{eq:Leib}
  (-1)^n \zeta^{(n)} \left(1-s,\frac{u}{v}  \right) = 2\sum_{a+b+c=n \atop
  a,b,c \geq 0} \left\{  \begin{pmatrix}
  n \\
  a,b,c\\
  \end{pmatrix}  \Gamma^{(a)}(s) \left( (2\pi v)^{-s} \right)^{(b)} \right. \nonumber \\
  \hspace{3cm} \times H^{(c)} \left(s,\frac{u}{v} \right) \biggr\},
\end{eqnarray}
where $ \begin{pmatrix}
n \\
a,b,c\\
\end{pmatrix} :=n!/(a!b!c!)$.
Consider the values of $\Gamma^{(a)}(s),\ ( (2\pi v)^{-s})^{(b)}$ and $H^{(c)} \left( s,\frac{u}{v} \right)$ in (\ref{eq:Leib})
at positive integers $s$. Let $m$ be a positive integer. First, we have
\begin{equation}\label{eq:s=m}
( (2\pi v)^{-s})^{(b)} \vline_{s=m} =(-\log{(2\pi v)})^b (2\pi v)^{-m}.
\end{equation}

Next, we consider the values of $H^{(c)} \left( m,\frac{u}{v} \right)$ for positive integers $m$. For $m\geq 2$, since
$\zeta  \left(s,\frac{k}{v} \right)$ is holomorphic at $s=m$, we have
\begin{align}
& H^{(c)} \left(m,\frac{u}{v} \right) \notag \\
& =\sum_{k=1}^v \sum_{i=0}^c \left\{ \begin{pmatrix}
c\\
i\\
\end{pmatrix} \cos{ \left(\frac{\pi s}{2}-\frac{2\pi ku}{v} \right)^{(c-i)}} \vline_{ \begin{array}{c}
\\
s=m \\
\end{array} } \times \zeta^{(i)} \left( m,\frac{k}{v} \right) \right\} \notag \\
&=-\sum_{k=1}^v \sum_{i=0}^c \left\{ \begin{pmatrix}
c\\
i
\end{pmatrix}  \times \left(\frac{\pi}{2} \right)^{c-i} \sin{ \left( \frac{\pi (m+c-i-1)}{2} -\frac{2\pi ku}{v} \right) }
 \times \zeta^{(i)} \left(m,\frac{k}{v} \right) \right\}, \notag \\
&  \zeta^{(i)} \left( m,\frac{k}{v} \right) = \sum_{\ell=0}^{\infty} \left(-\log{ \left( \ell+\frac{k}{v} \right)} \right)^i \left(\ell+\frac{k}{v} \right)^{-m}.
\notag
\end{align}

On the other hand, for  $m=1$, since
\begin{align*}
&\left\{ \cos{\left(\frac{\pi s}{2} -\frac{2\pi ku}{v} \right)} \sum_{n=0}^{\infty} \frac{(-1)^n}{n!} \gamma_n \left( \frac{k}{v} \right) (s-1)^n \right\}^{
(c)} \\
& = \sum_{i=0}^c \begin{pmatrix}
c\\
i\\
\end{pmatrix} \cos{ \left(\frac{\pi s}{2} -\frac{2\pi ku}{v} \right)^{(c-i)} }\left(
\sum_{n=0}^{\infty} \frac{(-1)^n}{n!} \gamma_n \left( \frac{k}{v} \right) (s-1)^n \right)^{(i)},
\end{align*}
and by (\ref{eq:cos=1}), we have
\begin{align}
h_{c,\frac{u}{v}} & := \sum_{k=1}^v \left\{ \cos{  \left( \frac{\pi s}{2}-\frac{2\pi ku}{v} \right)  } \sum_{n=0}^{\infty}
\frac{(-1)^n}{n!} \gamma_n \left( \frac{k}{v} \right) (s-1)^n \right\}^{(c)} \vline_{\begin{array}{c} \\s=1 \end{array}} \notag \\
& =\sum_{k=1}^v \sum_{i=0}^c (-1)^{i+1} \begin{pmatrix}
c\\
i\\
\end{pmatrix} \left( \frac{\pi}{2} \right)^{c-i} \sin{ \left( \frac{\pi (c-i)}{2}- \frac{2\pi ku}{v} \right)} \gamma_i \left(
\frac{k}{v} \right). \label{eq:h}
\end{align}
From (\ref{eq:H}),(\ref{eq:h}) and
\[
\sum_{k=1}^v \sin{ \frac{2\pi ku}{v} }=0,\quad \sum_{k=1}^v \cos{ \frac{2\pi ku}{v} } =\begin{cases}
0  & (u\ne v)  \\
v & (u=v)
\end{cases},
\]
we get
\begin{align}
H^{(c)} \left( 1,\frac{u}{v} \right) & =\varepsilon_{c,\frac{u}{v} } +h_{c,\frac{u}{v} }, \notag \\
\varepsilon_{c,\frac{u}{v}} & :=-\sum_{k=1}^v \frac{1}{c+1} \left( \frac{\pi}{2} \right)^{c+1} \sin{ \left( \frac{(c+1)\pi}{2}-
\frac{2\pi ku}{v} \right) } \notag \\
& = \begin{cases}
\dfrac{(-1)^{\frac{c}{2}+1}}{c+1} \left( \dfrac{\pi}{2} \right)^{c+1} v& (c\in 2\mathbb Z,\ u=v),\\
& \\
0& \text{(otherwise)}.
\end{cases} \notag
\end{align}
We established the following lemma.
\begin{lemma}\label{lem:Hc}
For $c\in \mathbb Z_{\geq 0},\ m\in \mathbb Z_{>0}$ and $y=u/v \in \mathbb Q$ with $1\leq u\leq v$, we  have
\[
H^{(c)}(m,y)= \left\{ \begin{array}{ll}
\displaystyle{
-\sum_{k=1}^v \sum_{i=0}^c \left\{ \begin{pmatrix}
c\\
i\\
\end{pmatrix} \times \left( \frac{\pi}{2} \right)^{c-i} \right. } & \\
\hspace{0.15cm}
\displaystyle{
\left. \times \sin{ \left( \frac{\pi (m+c-i-1)}{2} -2\pi ky \right) } \times \zeta^{(i)}\left( m,\frac{k}{v} \right) \right\} } & (m \geq 2), \\
\displaystyle{
\varepsilon_{c,y} +\sum_{k=1}^v \sum_{i=0}^c (-1)^{i+1} \begin{pmatrix}
c\\
i\\
\end{pmatrix} \left( \frac{\pi}{2} \right)^{c-i} } & \\
\hspace{2.8cm}
\displaystyle{
\times \sin{ \left( \frac{\pi (c-i)}{2} -2\pi ky \right)} \gamma_i \left( \frac{k}{v} \right) } & (m=1),
\end{array} \right.
\]
where
\begin{align*}
\zeta^{(i)} \left( m,\frac{k}{v} \right) & =\sum_{\ell=0}^{\infty} \left( -\log{\left( \ell+\frac{k}{v} \right)} \right)^i
\left( \ell+\frac{k}{v} \right)^{-m} \quad (m\geq 2), \\
& \\
\varepsilon_{c,y} &=\begin{cases}
\displaystyle{ \frac{(-1)^{\frac{c}{2}+1} }{c+1} \left( \frac{\pi}{2} \right)^{c+1} v} & (c\in 2\mathbb Z,\ y=1), \\
& \\
0 & (\text{otherwise}).
\end{cases}
\end{align*}
\end{lemma}

Finally we consider the value  $\Gamma^{(a)}(m)$ for  positive integers $m$. Coffey \cite{Co} pointed out that  the
value is given by using the complete Bell polynomials $\mathbb B_n$. For a positive integer $n$, we define
the complete Bell polynomial $\mathbb B_n=\mathbb B_n(X_1,\ldots,X_n)$ by
\[
\exp{ \left( \sum_{j=1}^{\infty} X_j \frac{t^j}{j!} \right) } =\sum_{n=0}^{\infty} \mathbb B_n( X_1,\ldots, X_n )\frac{t^n}{n!},
\]
and let $\mathbb B_0:=1$ (see \cite[Chap. III, \S 3.3]{Com}).
It is known that these polynomials are integral and have the following explicit formulae.
\begin{lemma}$($\cite[Chap. III, \S 3.3, Theorem~A]{Com}$)$ \label{lem:Com}
Let $n\geq 1$. We have
\[
\mathbb B_n (X_1, \ldots, X_n)\in \mathbb Z[X_1,\ldots,X_n]
\]
 and
\[
\mathbb B_n (X_1,\ldots, X_n)=\sum_{\pi (n)} \frac{n!}{k_1!k_s!\cdots k_n!} \left( \frac{X_1}{1!} \right)^{k_1} \left( \frac{X_1}{2!} \right)^{k_2}
\cdots \left( \frac{X_n}{n!} \right)^{k_n}
\]
where
\[
\pi (n) := \{ (k_1,k_2, \ldots,k_n )\in \mathbb Z^n \ | \ k_1, k_2,\ldots,k_n \geq 0,\ k_1+2k_2+\cdots +nk_n=n \}.
\]
\end{lemma}

We use $\mathbb B_n$ to give an explicit formula for the higher order derivatives of holomorphic functions.
\begin{lemma}$($\cite[p3, Lemma~1]{Co}$)$\label{lem:Cof}
Let $f(s)$ and $g(s)$ be holomorphic functions on an open subset ${\mathcal D}$ of $\mathbb C$ satisfying
$f'(s)=f(s)g(s)$. We have
\[
f^{(n)} (s) \ \left( := \left( \frac{d}{ds} \right)^n f(s) \right) =f(s) \ \mathbb B_n (g(s),g'(s),\ldots , g^{(n-1)}(s)).
\]
\end{lemma}

By applying Lemma~\ref{lem:Com} for the gamma function $\Gamma (s)$ and the psi function $\psi (s):=
\Gamma'(s)/\Gamma (s) $, we have
\begin{align*}
\Gamma^{(n)} (s) & =\Gamma(s) \ \mathbb B_n (\psi (s),\psi'(s),\ldots,\psi^{(n-1)} (s) ) \notag \\
& =\Gamma (s) \sum_{\pi (n) } \frac{n!}{k_1! k_2! \cdots k_n!} \left( \frac{\psi(s)}{1!} \right)^{k_1}
\left( \frac{\psi'(s)}{2!} \right)^{k_2} \cdots \left( \frac{\psi^{(n-1)} (s)}{n!} \right)^{k_n}.
\end{align*}
Since
\begin{align*}
\psi (1) & =\frac{\Gamma'(1)}{\Gamma (1)} =\int_0^{\infty} e^{-t} \log{t} \ dt =-\gamma, \\
\psi (s+1) & =\frac{\Gamma'(s+1)}{\Gamma (s+1)} =\frac{\Gamma (s)+s\Gamma'(s)}{s\Gamma (s) }=\frac{1}{s}+\psi (s),
\end{align*}
we have
\begin{equation*}\label{eq:psi-val}
\psi (m+1) =-\gamma +\sum_{k=1}^m \frac{1}{k} \qquad (m \in \mathbb Z_{\geq 1} ).
\end{equation*}

On the other hand, the value $\psi^{(n)} (m)$ for $n\geq 1$ is obtained as follows.
From the Weierstrass factorization theorem of the gamma function:
\[
\frac{1}{\Gamma (s)} =se^{\gamma s} \prod_{n=1}^{\infty} \left( 1+\frac{s}{n} \right)  e^{-\frac{s}{n} },
\]
we have
\[
\psi (s)= -\gamma +\sum_{n=1}^{\infty} \left( \frac{1}{n} -\frac{1}{s+n-1} \right),
\]
and hence
\[
\psi^{(n)} (s)=\sum_{k=1}^{\infty} \frac{(-1)^{n+1} n!} {(k+s-1)^{n+1}} \qquad (n \in \mathbb Z_{\geq 1}).
\]

Therefore, we obtain
\begin{equation*}\label{eq:dpsi}
\psi^{(n)} (m)=(-1)^{n+1} n! \left( \zeta (n+1)-\sum_{k=1}^{m-1} \frac{1}{k^{n+1}} \right)
\qquad (n,m \in \mathbb Z_{\geq 1}).
\end{equation*}
We established the following lemma.
\begin{lemma}\label{lem:higherG}
For $n\in \mathbb Z_{\geq 0}$ and $m\in \mathbb Z_{>0}$,  we  have
\begin{align*}
  &\Gamma^{(n)}(m) \\
   &= (m-1)! \sum_{\pi (n)} \frac{n!}{k_1! k_2!\cdots k_n!} \left( \frac{\psi (m)}{1!} \right)^{k_1} \left(
  \frac{\psi' (m)}{2!} \right)^{k_2} \cdots \left( \frac{\psi^{(n-1)}(m)}{n!} \right)^{k_n},
\end{align*}
where
\[\pi (n) = \{ (k_1,k_2, \ldots ,k_n ) \in \mathbb Z^n \ |\ k_1,k_2,\ldots, k_n \geq 0, \quad k_1+2k_2+\cdots+nk_n=n \},\]
\[
\psi^{(\ell)}(m) = \begin{cases}
\displaystyle{ -\gamma +\sum_{n=1}^{m-1} \frac{1}{k} }
 & (\ell=0),\\
 & \\
\displaystyle{ (-1)^{\ell+1} \ell ! \left( \zeta (\ell+1) -\sum_{k=1}^{m-1} \frac{1}{k^{\ell+1}} \right)} & (\ell \geq 1).
\end{cases}
\]
\end{lemma}
\begin{corollary}\label{cor:higherG1}
For $n\in \mathbb Z_{\geq 0}$, we  have
\[
\Gamma^{(n)}(1) = \sum_{\pi (n)} \frac{ (-1)^n n!}{k_1! k_2!\cdots k_n!} \left( \frac{\gamma}{1} \right)^{k_1} \left(
\frac{\zeta (2) }{2} \right)^{k_2}  \left(
\frac{\zeta (3) }{3} \right)^{k_3} \cdots \left( \frac{\zeta(n)}{n} \right)^{k_n}.
\]
\end{corollary}

Let $y$ be a positive rational number, and put $y=Y+u/v$ where $Y\in \mathbb Z_{\geq 0}$ and $u,v\in \mathbb Z,\ 1\leq  u \leq v$.
We have
\[
\zeta (s,y)=\zeta \left( s, Y+\frac{u}{v} \right) =\zeta \left(s,\frac{u}{v} \right)-\sum_{m=0}^{Y-1} \left( m+\frac{u}{v} \right)^{-s}.
\]
By differentiating, we have
\[
\zeta^{(j)} (s,y) = \zeta^{(j)} \left(s,\frac{u}{v} \right)-\sum_{m=0}^{Y-1} \left\{ -\log{\left(m+\frac{u}{v} \right)} \right\}^j \left(
m+\frac{u}{v} \right)^{-s}.
\]
From (\ref{eq:Leib}), we have
\begin{align}
\zeta^{(j)} (1-s,y)= &(-1)^j 2 \sum_{a+b+c=j \atop a,b,c >0} \left\{ \begin{pmatrix}
j \\
a,b,c \\
\end{pmatrix} \Gamma^{(a)} (s) \left( (2\pi v)^{-s} \right)^{(b)} H^{(c)} \left(s,\frac{u}{v} \right) \right\} \notag  \\
&  -\sum_{m=0}^{Y-1} \left\{ -\log{\left( m+\frac{u}{v} \right)} \right\}^j \left( m+\frac{u}{v} \right)^{s-1}. \label{eq:zeta^j}
\end{align}
Let $w_1=r_1/q_1,\ \ldots,\ w_N=r_N/q_N \ (r_i,q_i \in \mathbb Z_{>0} \ \text{for} \ i=1,2,\ldots,N)$ be positive rational numbers
with gcd$(r_i,q_i)=1$ for $i=1,2,\ldots,N$ and put
\[
w:= {\rm lcm}(r_1,\ldots,r_N) /{\rm gcd} (q_1,\ldots,q_N) ,\
\ell_1:=w/w_1, \ldots,\ell_N:=w/w_N.
\]
Furthermore, for a positive rational number $x$ and non-negative integers $k_1,\ldots$, $k_N$, we put
\begin{align*}
y(k_1,\ldots,k_N) & := w^{-1} (x+k_1w_1+\cdots+k_Nw_N) \\
& =Y(k_1,\ldots,k_N)+\frac{u(k_1,\ldots,k_N)}{v(k_1,\ldots,k_N)},
\end{align*}
where $Y(k_1,\ldots,k_N) \in \mathbb Z_{\geq 0}$ and $u(k_1,\ldots,k_N),\ v(k_1,\ldots,k_N) \in \mathbb Z$ with $1 \leq u(k_1,$ $\ldots, k_N) \leq v(k_1,\ldots ,k_N)$.
From (\ref{eq:higher-zeta0}) and (\ref{eq:zeta^j}), we have
\begin{align}
&\zeta_N^{(n)} (-s,x \ |\ w_1,\ldots, w_N) \\
&= \frac{1}{(N-1)!}
\sum_{j=0}^n \begin{pmatrix}
n\\
j\\
\end{pmatrix}  (-\log{w})^{n-j} w^s \sum_{k_1=0}^{\ell_1-1} \cdots \sum_{k_N=0}^{\ell_N-1} \sum_{k=0}^{N-1} \frac{C_{N,y(k_1,\ldots,k_N)}^{(k)} (0) }{k!} \notag\\
& \times
\left[ (-1)^j 2  \sum_{a+b+c=j \atop a,b,c>0} \left\{ \begin{pmatrix}
j\\
a,b,c\\
\end{pmatrix} \Gamma^{(a)} (s+k+1) \right. \right.  \notag \\
& \left.  \left. \times \left( (2\pi v(k_1,\ldots, k_N) )^{-(s+k+1)} \right)^{(b)} H^{(c)} \left( s+k+1, \frac{u(k_1,\ldots,k_N)}{v(k_1,\ldots,k_N)} \right) \right\} \right. \notag \\
& \left. -\sum_{m=0}^{Y(k_1,\ldots,k_N)-1} \left\{ -\log{\left( m+\frac{u(k_1,\ldots,k_N)}{v(k_1,\ldots,k_N)} \right) }\right\}^j \left( m+
\frac{u(k_1,\ldots,k_N)}{v(k_1,\ldots,k_N) } \right)^{s+k} \right] \label{eq:zeta_N^n}.
\end{align}

From  (\ref{eq:s=m}),  (\ref{eq:zeta_N^n}), and  Lemmas~\ref{lem:Hc}, \ref{lem:higherG}, we established the following theorem.
\begin{theorem}\label{theo:main}
Let $w_1=r_1/q_1,\ldots, w_N =r_N/q_N  \  (r_i,q_i \in \mathbb Z_{ >0}$ for $i=1,2, \ldots$, $N)$  be  positive  rational numbers
with ${\rm gcd}(r_i,q_i)=1$ for $i=1,\ldots,N$ and put
\[
w:= {\rm lcm}(r_1,\ldots,r_N) /{\rm gcd} (q_1,\ldots,q_N) ,\
\ell_1:=w/w_1, \ldots,\ell_N:=w/w_N.
\]
For any positive rational number $x$ and $\ell, n\in \mathbb Z_{\geq 0}$,
we have
\begin{align*}
& \zeta^{(n)}_N(-\ell ,x \ |\ w_1,\ldots,w_N) \\
& =\frac{1}{  (N-1)!}\sum_{j=0}^n \begin{pmatrix}
n\\
j\\
\end{pmatrix} (-\log{w})^{n-j} w^{\ell}   \sum_{k_1=0}^{\ell_1-1} \cdots \sum_{k_N=0}^{\ell_N-1}
\sum_{k=0}^{N-1} \frac{C^{(k)}_{N,y(k_1,\ldots,k_N)}(0)}{k!} \\
& \times  \left[ (-1)^j 2\sum_{a+b+c=j \atop  a,b,c \geq 0} \left\{ \begin{pmatrix}
j\\
a,b,c \\
\end{pmatrix}
 \Gamma^{(a)} (\ell+k+1) (-\log{(2\pi v(k_1,\ldots,k_N))})^b \right. \right. \\
 & \times
 \left. (2\pi v(k_1,\ldots,k_N))^{-(\ell+k+1)}  H^{(c)} \left(\ell+k+1, \frac{u(k_1,\ldots,k_N)}{v(k_1,\ldots,k_N)} \right) \right\} \\
 & \left. -  \sum_{m=0}^{Y(k_1,\ldots,k_N)-1} \left\{ -\log{ \left( m+\frac{u(k_1,\ldots,k_N)}{v(k_1,\ldots,k_N)} \right) } \right\}^j \left(
 m+\frac{u(k_1,\ldots,k_N)}{v(k_1,\ldots,k_N)} \right)^{\ell+k} \right],
\end{align*}
where
\begin{align*}
 y(k_1,\ldots,k_N) & := w^{-1} (x+k_1w_1+\cdots+k_Nw_N) \\
&  = Y (k_1,\ldots ,k_N)+u(k_1,\ldots,k_N)/v(k_1,\ldots,k_N),
\end{align*}
\begin{align*}
& Y(k_1,\ldots,k_N) \in \mathbb Z_{\geq 0}, \  u(k_1,\ldots,k_N), \ v(k_1,\ldots, k_N) \in \mathbb Z_{>0}, \\
& 1\leq u(k_1,\ldots, k_N) \leq v(k_1,\ldots, k_N),
\end{align*}
and $ H^{(c)}\left( \ell+k+1, u(k_1,\ldots, k_N)/v(k_1,\ldots, k_N) \right)$,
$\Gamma^{(a)}(\ell+k+1)$ are given by Lemmas \ref{lem:Hc}, \ref{lem:higherG}, respectively.
\end{theorem}

Especially, considering $\ell=0$, we get the following result.
\begin{corollary}\label{cor:main}
Let $w_1=r_1/q_1,\ldots, w_N =r_N/q_N  \  (r_i,q_i \in \mathbb Z_{ >0}$ for $i=1,2, \ldots$, $N)$  be  positive  rational numbers
with ${\rm gcd}(r_i,q_i)=1$ for $i=1,\ldots,N$ and put
\[
w:= {\rm lcm}(r_1,\ldots,r_N) /{\rm gcd} (q_1,\ldots,q_N) ,\
\ell_1:=w/w_1, \ldots,\ell_N:=w/w_N.
\]
For any positive rational number $x$  and $n\in \mathbb Z_{\geq 0}$,
we have
\begin{align*}
& \zeta^{(n)}_N(0,x \ |\ w_1,\ldots,w_N) \\
& =\frac{1}{  (N-1)!}\sum_{j=0}^n \begin{pmatrix}
n\\
j\\
\end{pmatrix} (-\log{w})^{n-j}  \sum_{k_1=0}^{\ell_1-1} \cdots \sum_{k_N=0}^{\ell_N-1}
\sum_{k=0}^{N-1} \frac{C^{(k)}_{N,y(k_1,\ldots,k_N)}(0)}{k!} \\
& \times \left[ (-1)^j 2  \sum_{a+b+c=j \atop  a,b,c \geq 0} \left\{ \begin{pmatrix}
j\\
a,b,c \\
\end{pmatrix}
 \Gamma^{(a)} (k+1) (-\log{(2\pi v(k_1,\ldots,k_N))})^b \right. \right. \\
 & \left. \times  (2\pi v(k_1,\ldots,k_N))^{-(k+1)}  H^{(c)} \left( k+1, \frac{u(k_1,\ldots, k_N)}{v(k_1,\ldots, k_N)} \right) \right\}\\
 &  \left. -  \sum_{m=0}^{Y(k_1,\ldots,k_N)-1} \left\{ -\log{ \left( m+\frac{u(k_1,\ldots,k_N)}{v(k_1,\ldots,k_N)} \right) } \right\}^j \left(
 m+\frac{u(k_1,\ldots,k_N)}{v(k_1,\ldots,k_N)} \right)^{k} \right],
\end{align*}
where
\begin{align*}
 y(k_1,\ldots,k_N) & := w^{-1} (x+k_1w_1+\cdots+k_Nw_N) \\
&  = Y (k_1,\ldots ,k_N)+u(k_1,\ldots,k_N)/v(k_1,\ldots,k_N),
\end{align*}
\begin{align*}
& Y(k_1,\ldots,k_N) \in \mathbb Z_{\geq 0}, \  u(k_1,\ldots,k_N), \ v(k_1,\ldots, k_N) \in \mathbb Z_{>0}, \\
& 1\leq u(k_1,\ldots, k_N) \leq v(k_1,\ldots, k_N),
\end{align*}
and $ H^{(c)}\left( k+1, u(k_1,\ldots, k_N)/v(k_1,\ldots, k_N) \right)$,
$\Gamma^{(a)}(k+1)$ are given by Lemmas \ref{lem:Hc}, \ref{lem:higherG}, respectively.
\end{corollary}

The figures below show graphs of $\zeta_2(0,x \ |\ w_1,w_2 )$
with $x=1/10,1/20,1/30$ and $0<w_1 <1,\ 0<w_2<1$ drawn using
Corollary~\ref{cor:main}.

\begin{figure}[H]
\begin{center}  \vspace{-0.5cm}
\includegraphics[width=9.5cm]{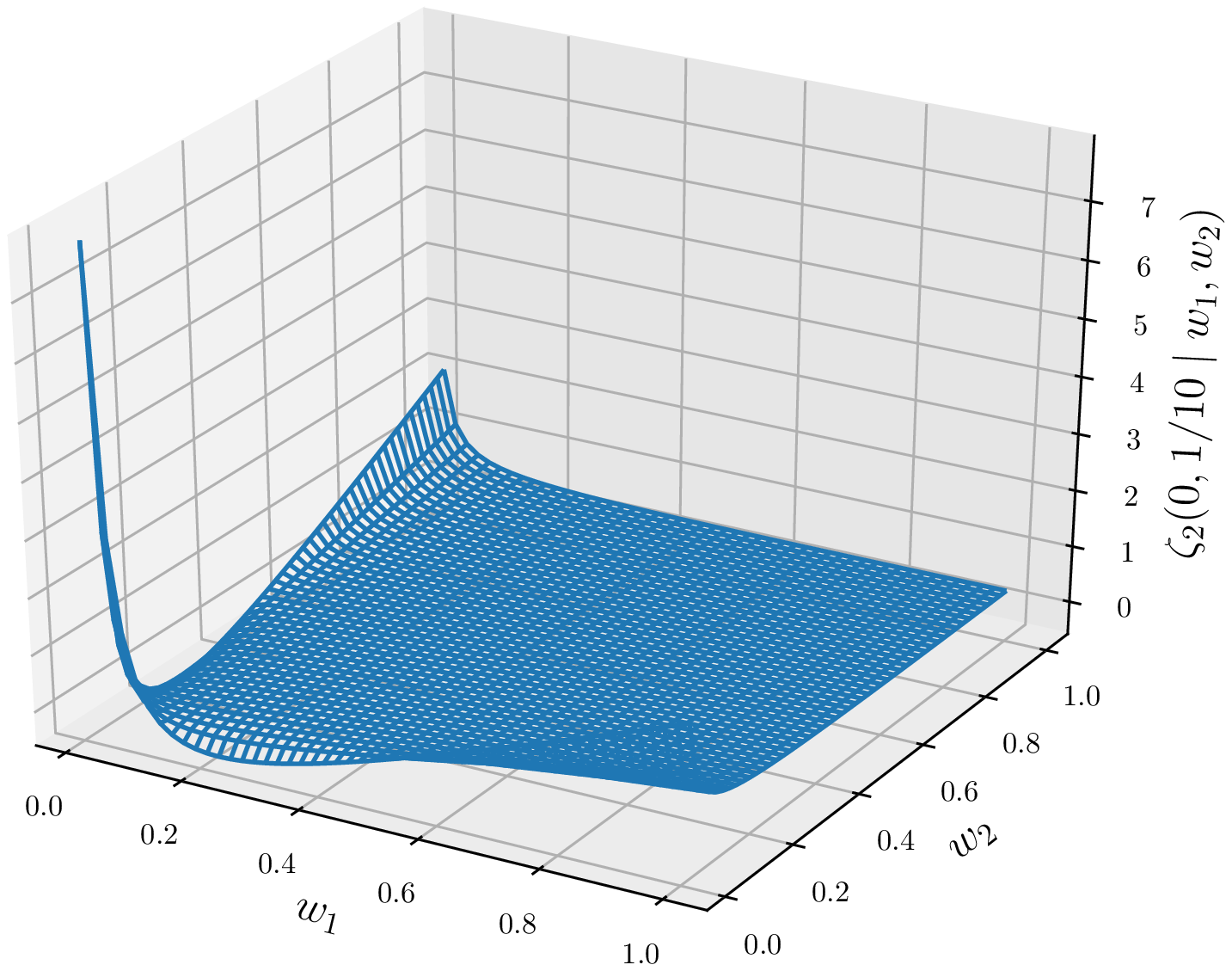}
\end{center}
\caption{$\zeta_2(0,1/10\ |\ w_1,w_2)$}
\label{fig:bzeta}
\end{figure}

\begin{figure}[H]
\begin{center}
\includegraphics[width=9.5cm]{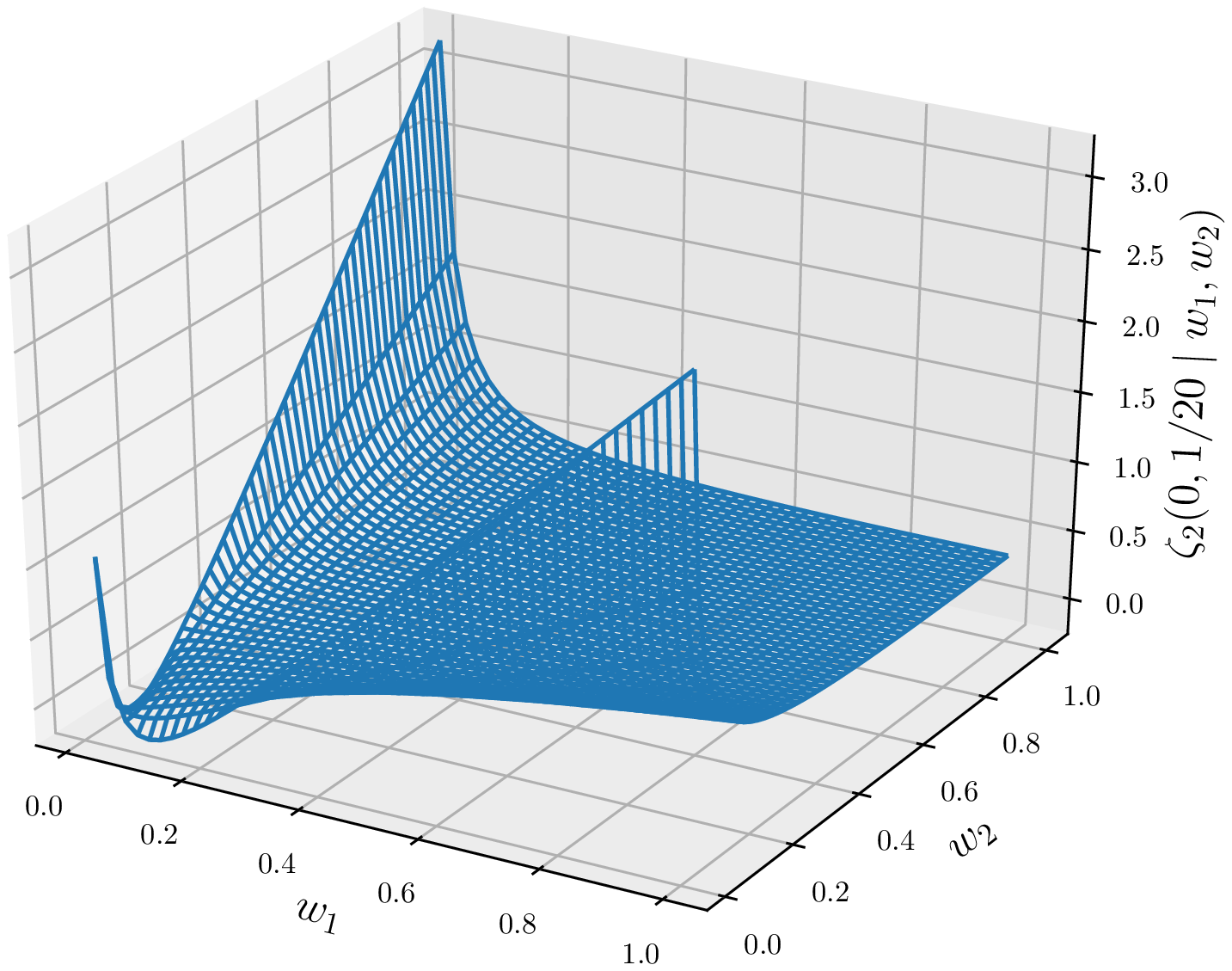}
\end{center}
\caption{$\zeta_2(0,1/20\ |\ w_1,w_2)$}
\label{fig:bzeta_2}
\end{figure}

\begin{figure}[H]
\begin{center}
\includegraphics[width=9.5cm]{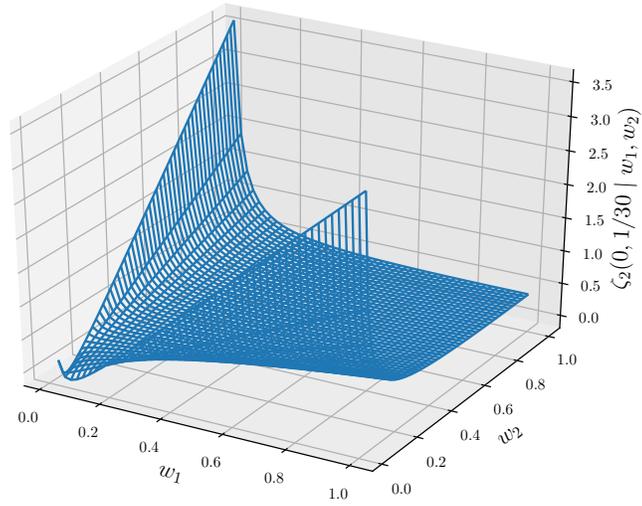}
\end{center}
\caption{$\zeta_2(0,1/30\ |\ w_1,w_2)$}
\label{fig:bzeta_3}
\end{figure}

The figures below show graphs of $\zeta_2'(0,x \ |\ w_1,w_2 )$
with $x=1/10,1/20,1/30$ and $0<w_1 <1,\ 0<w_2<1$ drawn using
Corollary~\ref{cor:main}.

\begin{figure}[H]
\begin{center}
\includegraphics[width=9.5cm]{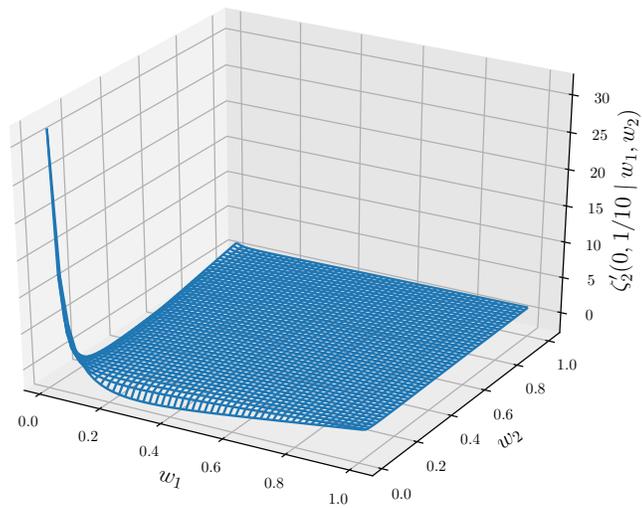}
\end{center}
\caption{$\zeta_2'(0,1/10\ |\ w_1,w_2)$}
\label{fig:bzeta'}
\end{figure}

\begin{figure}[H]
\begin{center}
\includegraphics[width=9.5cm]{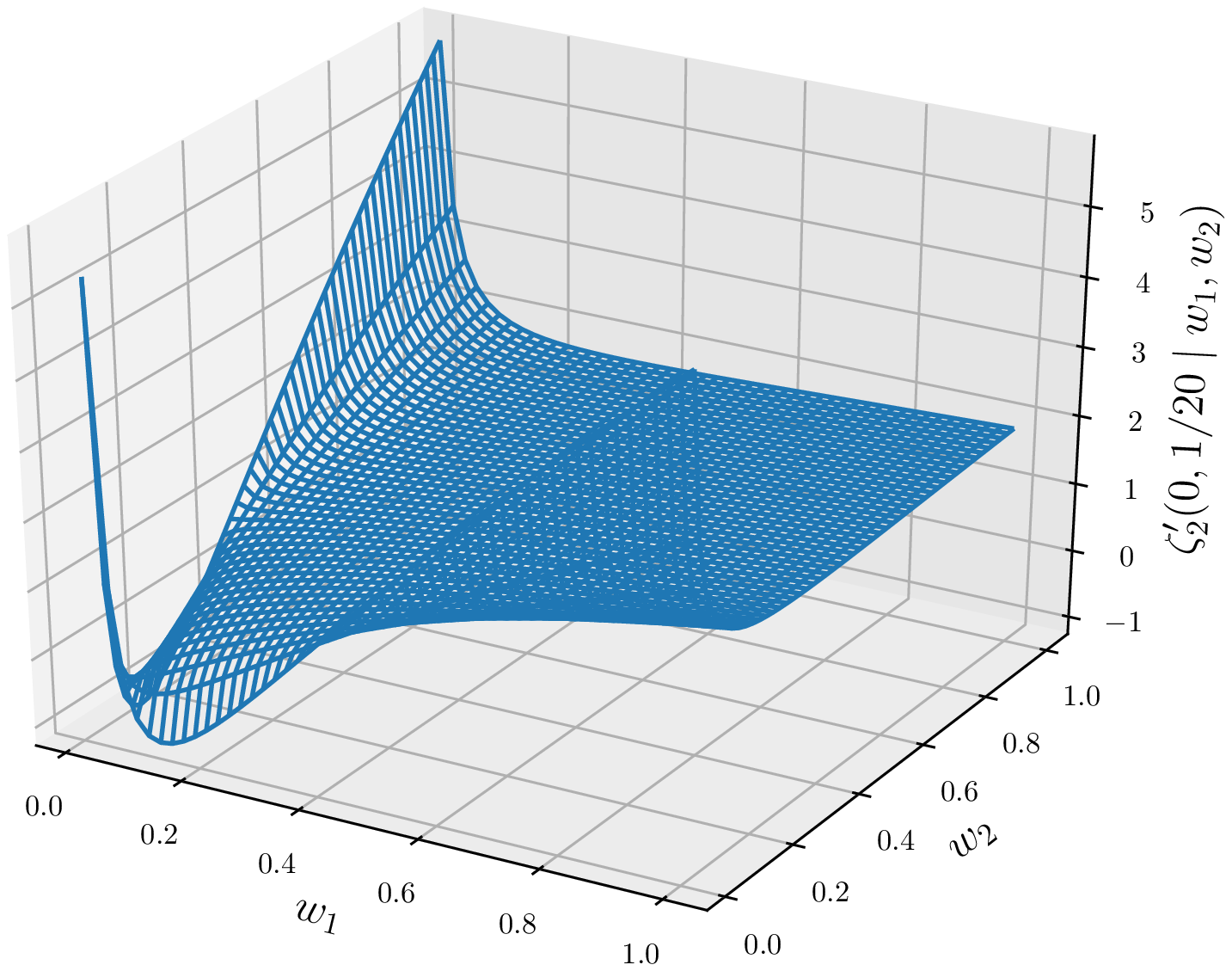}
\end{center}
\caption{$\zeta_2'(0,1/20\ |\ w_1,w_2)$}
\label{fig:bzeta_2'}
\end{figure}

\begin{figure}[H]
\begin{center}
\includegraphics[width=9.5cm]{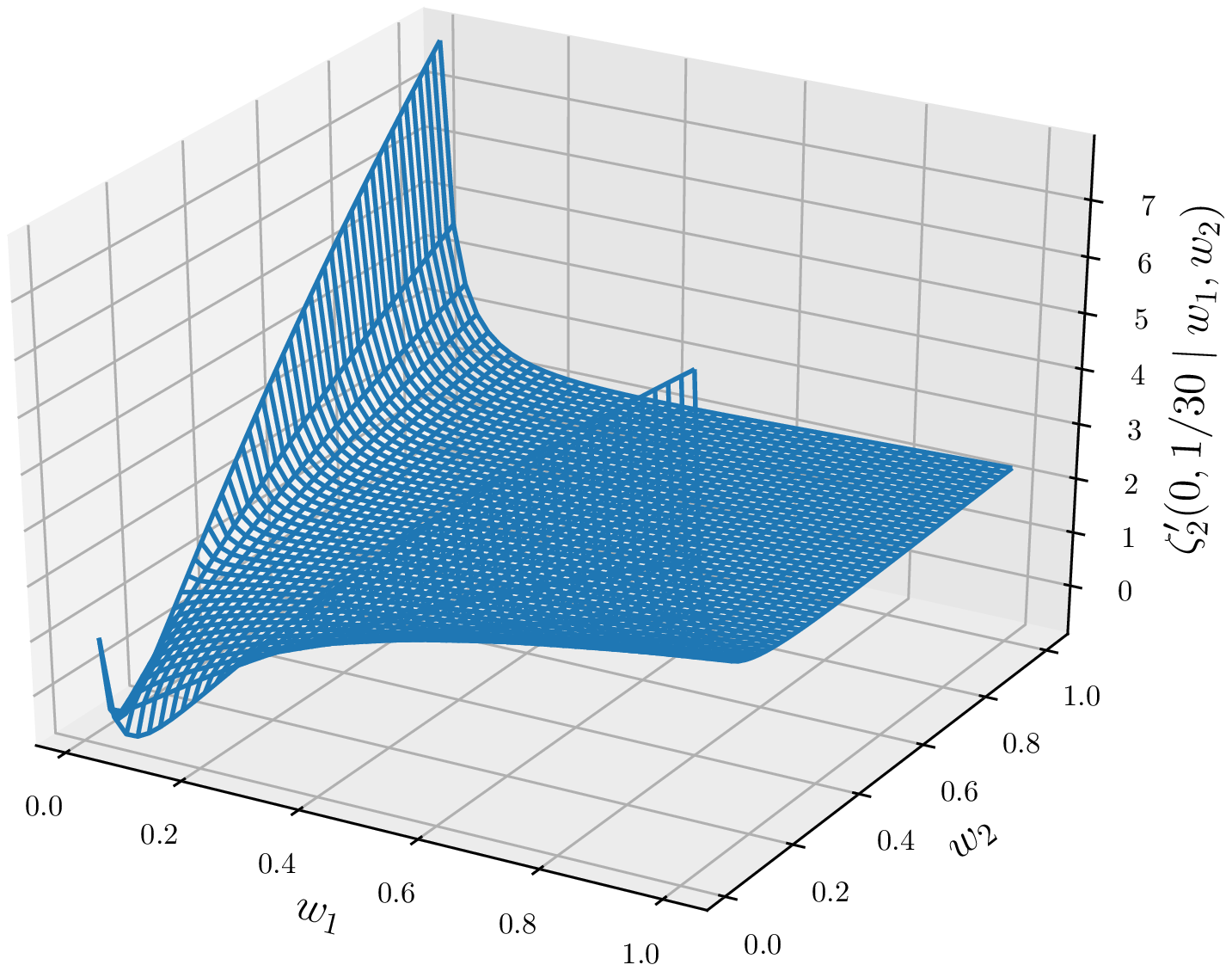}
\end{center}
\caption{$\zeta_2'(0,1/30\ |\ w_1,w_2)$}
\label{fig:bzeta_3'}
\end{figure}

By applying Corollary~\ref{cor:main} for $w_1=\cdots =w_N=1$, we get the values of $\zeta_N^{(n)} (0,x)=
\zeta_N^{(n)} (0,x \ | \ 1,\ldots,1)$.
\begin{corollary}\label{cor:zN=0}
For any positive rational number $x$, we put $x=X+u/v$ with $X\in \mathbb Z_{\geq 0},\ u,v\in \mathbb Z,\ 1\leq u\leq v$.
For  $n\in \mathbb Z_{\geq 0}$, we have
\begin{align*}
 &\zeta^{(n)}_N(0,x) =\frac{1}{(N-1)!}
\sum_{k=0}^{N-1} \frac{C^{(k)}_{N,x}(0)}{k!}  \\
&\times \left[ (-1)^n 2
 \sum_{a+b+c=n \atop  a,b,c \geq 0} \left\{ \begin{pmatrix}
n\\
a,b,c \\
\end{pmatrix}
 \Gamma^{(a)} (k+1) (-\log{(2\pi v)})^b  \right. \right. \\
 &  \left.  \left. \times(2\pi v)^{-(k+1)}
  H^{(c)} \left(k+1,\frac{u}{v} \right) \right\}
   -\sum_{m=0}^{X-1} \left\{ -\log{ \left(m+\frac{u}{v} \right)} \right\}^n \left(m+\frac{u}{v} \right)^k \right],
\end{align*}
where $ H^{(c)}(k+1, u/v)$
 and $\Gamma^{(a)}(k+1)$ are given by Lemmas \ref{lem:Hc}  and \ref{lem:higherG}, respectively.
\end{corollary}

Since $\zeta (s,x)=\zeta_1(s,x \ |\ 1)=\zeta_1(s,x)$ and $\zeta(s)=\zeta(s,1)$, we obtain the values at $s=0$ of higher order derivatives of the
Hurwitz zeta function and the Riemann zeta function.

\begin{corollary}\label{cor:Hur0}
For any positive rational number, we put  $x=X+u/v$ with $X\in \mathbb Z_{\geq 0},\ u,v\in \mathbb Z, \ 1\leq u\leq v$.
For $n\in \mathbb Z_{\geq 0}$, we have
\begin{align*}
 \zeta^{(n)}(0,x)  = &\frac{(-1)^n}{ \pi v }  \sum_{a+b+c=n \atop  a,b,c \geq 0} \left\{ \begin{pmatrix}
n\\
a,b,c \\
\end{pmatrix}
 \Gamma^{(a)} (1) (-\log{(2\pi v)})^b
  H^{(c)} \left(1,\frac{u}{v} \right) \right\} \\
  & - \sum_{m=0}^{X-1} \left\{ -\log{ \left( m+\frac{u}{v} \right)} \right\}^n,
\end{align*}
where
\[
\Gamma^{(a)}(1) = \sum_{\pi (a)} \frac{(-1)^a a!}{k_1! k_2!\cdots k_a!} \left( \frac{\gamma}{1} \right)^{k_1} \left(
\frac{\zeta (2) }{2} \right)^{k_2}  \left(
\frac{\zeta (3) }{3} \right)^{k_3} \cdots \left( \frac{\zeta (a) }{a} \right)^{k_a},
\]
and
\begin{align*}
  &H^{(c)}\left(1, \frac{u}{v} \right) \\
  &= \varepsilon_{c,\frac{u}{v}} +\sum_{k=1}^v \sum_{i=0}^c (-1)^{i+1} \begin{pmatrix}
  c\\
  i\\
  \end{pmatrix} \left( \frac{\pi}{2} \right)^{c-i} \sin{ \left( \frac{\pi (c-i)}{2} -\frac{2\pi ku}{v} \right)} \gamma_i \left( \frac{k}{v} \right),
\end{align*}
\begin{align*}
\varepsilon_{c,\frac{u}{v}} &=\begin{cases}
\displaystyle{ \frac{(-1)^{\frac{c}{2}+1} }{c+1} \left( \frac{\pi}{2} \right)^{c+1} v} & (c\in 2\mathbb Z,\ u=v), \\
& \\
0 & (\text{otherwise}).
\end{cases}
\end{align*}
\end{corollary}

\begin{example}\label{ex:zeta=0}
Let $u,v \in \mathbb Z$ with $1\leq u\leq v$.
\begin{align*}
  \zeta \left( 0,\frac{u}{v} \right) &= \begin{cases}
  -\dfrac{1}{2}  & (u=v) \\
  \displaystyle{  \frac{1}{\pi v} \sum_{k=1}^v \gamma_0 \left(\frac{k}{v} \right) \sin{ \left( \frac{2\pi ku}{v} \right) }  } & (u\ne v)
  \end{cases} \\
  & \\
  \zeta' \left( 0,\frac{u}{v} \right) &= \left\{
  \begin{array}{ll}
    \displaystyle{ -\dfrac{1}{2} (\gamma +\log{(2\pi v)})+\frac{1}{2v} \sum_{k=1}^v \gamma_0 \left( \frac{k}{v} \right) } & (u=v) \\
    \displaystyle{  \frac{1}{\pi v} \left\{ (\gamma+\log{(2\pi v)}) \sum_{k=1}^v \gamma_0 \left(\frac{k}{v} \right) \sin{ \left( \frac{2\pi ku}{v} \right) }   \right. } & \\
    \displaystyle{ \left. +\sum_{k=1}^v \frac{\pi}{2} \gamma_0 \left( \frac{k}{v} \right) \cos{ \left( \frac{2\pi ku}{v} \right)} +\gamma_1 \left( \frac{k}{v} \right)
    \sin{ \left(\frac{2\pi ku}{v} \right) } \right\}  } & (u\ne v)
  \end{array}
  \right.
\end{align*}
\end{example}

The equations $\zeta (0,x)=\frac{1}{2}-x$ and $\zeta' (0,x) \left( =\frac{ \partial}{ \partial s} \zeta (s,x) \vline_{\hspace{0.5pt} s=0} \right)
=\log{ \left( \frac{\Gamma (x) }{\sqrt{2\pi}} \right)}$ (\cite{L}) are known, but similar simple equations for the higher
order derivatives $\zeta^{(n)} (0,x) =\frac{\partial^n}{\partial s^n} \zeta (s,x) \vline_{s=0}$ are unknown.

\begin{corollary}\label{cor:Riemann0}
For any  $n\in \mathbb Z_{\geq 0}$, we have
\begin{align*}
 \zeta^{(n)}(0)  =\frac{(-1)^n}{ \pi  }  \sum_{a+b+c=n \atop  a,b,c \geq 0} \left\{ \begin{pmatrix}
n\\
a,b,c \\
\end{pmatrix}
 \Gamma^{(a)} (1) (-\log{(2\pi )})^b
  H^{(c)} (1,1) \right\},
\end{align*}
where
\[
\Gamma^{(a)}(1) = \sum_{\pi (a)} \frac{(-1)^a a!}{k_1! k_2!\cdots k_a!} \left( \frac{\gamma}{1} \right)^{k_1} \left(
\frac{\zeta (2) }{2} \right)^{k_2}  \left(
\frac{\zeta (3) }{3} \right)^{k_3} \cdots \left( \frac{\zeta (a) }{a} \right)^{k_a},
\]
and
\[
H^{(c)}(1,1)= \varepsilon_{c,1} + \sum_{i=0}^c (-1)^{i+1} \begin{pmatrix}
c\\
i\\
\end{pmatrix} \left( \frac{\pi}{2} \right)^{c-i}  \gamma_i \ \sin{  \frac{\pi (c-i)}{2}},
\]
\begin{align*}
\varepsilon_{c,1} &=\begin{cases}
\displaystyle{ \frac{(-1)^{\frac{c}{2}+1} }{c+1} \left( \frac{\pi}{2} \right)^{c+1} } & (c\in 2\mathbb Z),  \vspace{2mm} \\
0 & (c\not\in 2\mathbb Z).
\end{cases}
\end{align*}
\end{corollary}

\begin{example}\label{ex:zeta}
\begin{align*} \vspace*{-0.8cm}
& \zeta (0)=-\frac{1}{2} \\
& \zeta'(0)=-\frac{1}{2} \log{(2\pi)} \\
& \zeta''(0)=\frac{\pi^3}{24}-\frac{1}{2} \log^2{(2\pi)}-\frac{1}{2} \zeta (2)+\frac{1}{2} \gamma^2-\gamma_1
\end{align*}
\end{example}

Apostol~\cite[Theorem~3]{A2} has given another formula for $\zeta^{(n)} (0)$ involving the
coefficients of the Laurent expansion of $\Gamma (s) \zeta (s)$ at $s=1$.

\begin{longtable}[H]{l}
 \caption{$\Gamma^{(n)}(s)$ for $n=0,\ldots ,10$}
 \label{table:dgamma-s}
 \endfirsthead
 \endhead
 \vspace{1mm}
 $\Gamma^{(0)}(s)  = \Gamma(s)$, \\ \vspace{1mm}
 $\Gamma^{(1)}(s) = \Gamma(s)\psi(s)$, \\ \vspace{1mm}
 $\Gamma^{(2)}(s) = \Gamma(s)(\psi(s)^{2} + \psi^{(1)}(s))$, \\ \vspace{1mm}
 $\Gamma^{(3)}(s) = \Gamma(s)(\psi(s)^{3} + 3\psi(s)\psi^{(1)}(s) + \psi^{(2)}(s))$, \\ \vspace{1mm}
 $\Gamma^{(4)}(s) = \Gamma(s)(\psi(s)^{4} + 6\psi(s)^{2}\psi^{(1)}(s) + 4\psi(s)\psi^{(2)}(s) + 3\psi^{(1)}(s)^{2} + \psi^{(3)}(s))$, \\ \vspace{1mm}
 $\Gamma^{(5)}(s)$ \\ \vspace{1mm}
 $= \Gamma(s)(\psi(s)^{5} + 10\psi(s)^{3}\psi^{(1)}(s) + 10\psi(s)^{2}\psi^{(2)}(s) + 15\psi(s)\psi^{(1)}(s)^{2} + 5\psi(s)\psi^{(3)}(s)$ \\ \vspace{1mm}
 $ + 10\psi^{(1)}(s)\psi^{(2)}(s) + \psi^{(4)}(s))$, \\ \vspace{1mm}
 $\Gamma^{(6)}(s)$ \\ \vspace{1mm}
 $= \Gamma(s)(\psi(s)^{6} + 15\psi(s)^{4}\psi^{(1)}(s) + 20\psi(s)^{3}\psi^{(2)}(s) + 45\psi(s)^{2}\psi^{(1)}(s)^{2}$ \\ \vspace{1mm}
 $+ 15\psi(s)^{2}\psi^{(3)}(s) + 60\psi(s)\psi^{(1)}(s)\psi^{(2)}(s) + 6\psi(s)\psi^{(4)}(s) + 15\psi^{(1)}(s)^{3}$ \\ \vspace{1mm}
 $ + 15\psi^{(1)}(s)\psi^{(3)}(s) + 10\psi^{(2)}(s)^{2} + \psi^{(5)}(s))$, \\ \vspace{1mm}
 $\Gamma^{(7)}(s)$ \\ \vspace{1mm}
 $= \Gamma(s)(\psi(s)^{7} + 21\psi(s)^{5}\psi^{(1)}(s) + 35\psi(s)^{4}\psi^{(2)}(s) + 105\psi(s)^{3}\psi^{(1)}(s)^{2}$ \\ \vspace{1mm}
 $+ 35\psi(s)^{3}\psi^{(3)}(s) + 210\psi(s)^{2}\psi^{(1)}(s)\psi^{(2)}(s) + 21\psi(s)^{2}\psi^{(4)}(s) + 105\psi(s)\psi^{(1)}(s)^{3}$ \\ \vspace{1mm}
 $+ 105\psi(s)\psi^{(1)}(s)\psi^{(3)}(s) + 70\psi(s)\psi^{(2)}(s)^{2} + 7\psi(s)\psi^{(5)}(s) + 105\psi^{(1)}(s)^{2}\psi^{(2)}(s)$ \\ \vspace{1mm}
 $+ 21\psi^{(1)}(s)\psi^{(4)}(s) + 35\psi^{(2)}(s)\psi^{(3)}(s) + \psi^{(6)}(s))$, \\ \vspace{1mm}
 $\Gamma^{(8)}(s)$ \\ \vspace{1mm}
 $= \Gamma(s)(\psi(s)^{8} + 28\psi(s)^{6}\psi^{(1)}(s) + 56\psi(s)^{5}\psi^{(2)}(s) + 210\psi(s)^{4}\psi^{(1)}(s)^{2}$ \\ \vspace{1mm}
 $+ 70\psi(s)^{4}\psi^{(3)}(s) + 560\psi(s)^{3}\psi^{(1)}(s)\psi^{(2)}(s) + 56\psi(s)^{3}\psi^{(4)}(s) + 420\psi(s)^{2}\psi^{(1)}(s)^{3}$ \\ \vspace{1mm}
 $+ 420\psi(s)^{2}\psi^{(1)}(s)\psi^{(3)}(s) + 280\psi(s)^{2}\psi^{(2)}(s)^{2} + 28\psi(s)^{2}\psi^{(5)}(s)$ \\ \vspace{1mm}
 $+ 840\psi(s)\psi^{(1)}(s)^{2}\psi^{(2)}(s) + 168\psi(s)\psi^{(1)}(s)\psi^{(4)}(s) + 280\psi(s)\psi^{(2)}(s)\psi^{(3)}(s)$ \\ \vspace{1mm}
 $+ 8\psi(s)\psi^{(6)}(s) + 105\psi^{(1)}(s)^{4} + 210\psi^{(1)}(s)^{2}\psi^{(3)}(s) + 280\psi^{(1)}(s)\psi^{(2)}(s)^{2}$ \\ \vspace{1mm}
 $+ 28\psi^{(1)}(s)\psi^{(5)}(s) + 56\psi^{(2)}(s)\psi^{(4)}(s) + 35\psi^{(3)}(s)^{2} + \psi^{(7)}(s))$, \\ \vspace{1mm}
 $\Gamma^{(9)}(s)$ \\ \vspace{1mm}
 $= \Gamma(s)(\psi(s)^{9} + 36\psi(s)^{7}\psi^{(1)}(s) + 84\psi(s)^{6}\psi^{(2)}(s) + 378\psi(s)^{5}\psi^{(1)}(s)^{2}$ \\ \vspace{1mm}
 $+ 126\psi(s)^{5}\psi^{(3)}(s) + 1260\psi(s)^{4}\psi^{(1)}(s)\psi^{(2)}(s) + 126\psi(s)^{4}\psi^{(4)}(s)$ \\ \vspace{1mm}
 $+ 1260\psi(s)^{3}\psi^{(1)}(s)^{3} + 1260\psi(s)^{3}\psi^{(1)}(s)\psi^{(3)}(s) + 840\psi(s)^{3}\psi^{(2)}(s)^{2}$ \\ \vspace{1mm}
 $+ 84\psi(s)^{3}\psi^{(5)}(s) + 3780\psi(s)^{2}\psi^{(1)}(s)^{2}\psi^{(2)}(s) + 756\psi(s)^{2}\psi^{(1)}(s)\psi^{(4)}(s)$ \\ \vspace{1mm}
 $+ 1260\psi(s)^{2}\psi^{(2)}(s)\psi^{(3)}(s) + 36\psi(s)^{2}\psi^{(6)}(s) + 945\psi(s)\psi^{(1)}(s)^{4}$ \\ \vspace{1mm}
 $+ 1890\psi(s)\psi^{(1)}(s)^{2}\psi^{(3)}(s) + 2520\psi(s)\psi^{(1)}(s)\psi^{(2)}(s)^{2} + 252\psi(s)\psi^{(1)}(s)\psi^{(5)}(s)$ \\ \vspace{1mm}
 $+ 504\psi(s)\psi^{(2)}(s)\psi^{(4)}(s) + 315\psi(s)\psi^{(3)}(s)^{2} + 9\psi(s)\psi^{(7)}(s) + 1260\psi^{(1)}(s)^{3}\psi^{(2)}(s)$ \\ \vspace{1mm}
 $+ 378\psi^{(1)}(s)^{2}\psi^{(4)}(s) + 1260\psi^{(1)}(s)\psi^{(2)}(s)\psi^{(3)}(s) + 36\psi^{(1)}(s)\psi^{(6)}(s) + 280\psi^{(2)}(s)^{3}$ \\ \vspace{1mm}
 $+ 84\psi^{(2)}(s)\psi^{(5)}(s) + 126\psi^{(3)}(s)\psi^{(4)}(s) + \psi^{(8)}(s))$, \\ \vspace{1mm}
 $\Gamma^{(10)}(s)$ \\ \vspace{1mm}
 $= \Gamma(s)(\psi(s)^{10} + 45\psi(s)^{8}\psi^{(1)}(s) + 120\psi(s)^{7}\psi^{(2)}(s) + 630\psi(s)^{6}\psi^{(1)}(s)^{2}$ \\ \vspace{1mm}
 $+ 210\psi(s)^{6}\psi^{(3)}(s) + 2520\psi(s)^{5}\psi^{(1)}(s)\psi^{(2)}(s) + 252\psi(s)^{5}\psi^{(4)}(s)$ \\ \vspace{1mm}
 $+ 3150\psi(s)^{4}\psi^{(1)}(s)^{3} + 3150\psi(s)^{4}\psi^{(1)}(s)\psi^{(3)}(s) + 2100\psi(s)^{4}\psi^{(2)}(s)^{2}$ \\ \vspace{1mm}
 $+ 210\psi(s)^{4}\psi^{(5)}(s) + 12600\psi(s)^{3}\psi^{(1)}(s)^{2}\psi^{(2)}(s) + 2520\psi(s)^{3}\psi^{(1)}(s)\psi^{(4)}(s)$ \\ \vspace{1mm}
 $+ 4200\psi(s)^{3}\psi^{(2)}(s)\psi^{(3)}(s) + 120\psi(s)^{3}\psi^{(6)}(s) + 4725\psi(s)^{2}\psi^{(1)}(s)^{4}$ \\ \vspace{1mm}
 $+ 9450\psi(s)^{2}\psi^{(1)}(s)^{2}\psi^{(3)}(s) + 12600\psi(s)^{2}\psi^{(1)}(s)\psi^{(2)}(s)^{2}$ \\ \vspace{1mm}
 $+ 1260\psi(s)^{2}\psi^{(1)}(s)\psi^{(5)}(s) + 2520\psi(s)^{2}\psi^{(2)}(s)\psi^{(4)}(s) + 1575\psi(s)^{2}\psi^{(3)}(s)^{2}$ \\ \vspace{1mm}
 $+ 45\psi(s)^{2}\psi^{(7)}(s) + 12600\psi(s)\psi^{(1)}(s)^{3}\psi^{(2)}(s) + 3780\psi(s)\psi^{(1)}(s)^{2}\psi^{(4)}(s)$ \\ \vspace{1mm}
 $+ 12600\psi(s)\psi^{(1)}(s)\psi^{(2)}(s)\psi^{(3)}(s) + 360\psi(s)\psi^{(1)}(s)\psi^{(6)}(s) + 2800\psi(s)\psi^{(2)}(s)^{3}$ \\ \vspace{1mm}
 $+ 840\psi(s)\psi^{(2)}(s)\psi^{(5)}(s) + 1260\psi(s)\psi^{(3)}(s)\psi^{(4)}(s) + 10\psi(s)\psi^{(8)}(s) + 945\psi^{(1)}(s)^{5}$ \\ \vspace{1mm}
 $+ 3150\psi^{(1)}(s)^{3}\psi^{(3)}(s) + 6300\psi^{(1)}(s)^{2}\psi^{(2)}(s)^{2} + 630\psi^{(1)}(s)^{2}\psi^{(5)}(s)$ \\ \vspace{1mm}
 $+ 2520\psi^{(1)}(s)\psi^{(2)}(s)\psi^{(4)}(s) + 1575\psi^{(1)}(s)\psi^{(3)}(s)^{2} + 45\psi^{(1)}(s)\psi^{(7)}(s)$ \\ \vspace{1mm}
 $+ 2100\psi^{(2)}(s)^{2}\psi^{(3)}(s) + 120\psi^{(2)}(s)\psi^{(6)}(s) + 210\psi^{(3)}(s)\psi^{(5)}(s) + 126\psi^{(4)}(s)^{2}$ \\ \vspace{1mm}
 $+ \psi^{(9)}(s))$. \vspace{1mm}
\end{longtable}

\begin{longtable}[H]{l}
\caption{$\Gamma^{(n)}(1)$ for $n=0,\ldots , 10$}
\label{table:dgamma}
\endfirsthead
\endhead
\vspace{1mm}
$\Gamma^{(0)}(1) = 1$, \\ \vspace{1mm}
$\Gamma^{(1)}(1) = -\gamma$, \\ \vspace{1mm}
$\Gamma^{(2)}(1) = \gamma^{2} + \zeta(2)$, \\ \vspace{1mm}
$\Gamma^{(3)}(1) = -\gamma^{3} - 3\gamma\zeta(2) - 2\zeta(3)$, \\ \vspace{1mm}
$\Gamma^{(4)}(1) = \gamma^{4} + 6\gamma^{2}\zeta(2) + 8\gamma\zeta(3) + 3\zeta(2)^{2} + 6\zeta(4)$, \\ \vspace{1mm}
$\Gamma^{(5)}(1)$ \\ \vspace{1mm}
$= -\gamma^{5} - 10\gamma^{3}\zeta(2) - 20\gamma^{2}\zeta(3) - 15\gamma\zeta(2)^{2} - 30\gamma\zeta(4) - 20\zeta(2)\zeta(3) - 24\zeta(5)$, \\ \vspace{1mm}
$\Gamma^{(6)}(1)$ \\ \vspace{1mm}
$= \gamma^{6} + 15\gamma^{4}\zeta(2) + 40\gamma^{3}\zeta(3) + 45\gamma^{2}\zeta(2)^{2} + 90\gamma^{2}\zeta(4) + 120\gamma\zeta(2)\zeta(3)$ \\ \vspace{1mm}
$+ 144\gamma\zeta(5) + 15\zeta(2)^{3} + 90\zeta(2)\zeta(4) + 40\zeta(3)^{2} + 120\zeta(6)$, \\ \vspace{1mm}
$\Gamma^{(7)}(1)$ \\ \vspace{1mm}
$= -\gamma^{7} - 21\gamma^{5}\zeta(2) - 70\gamma^{4}\zeta(3) - 105\gamma^{3}\zeta(2)^{2} - 210\gamma^{3}\zeta(4)$ \\ \vspace{1mm}
$- 420\gamma^{2}\zeta(2)\zeta(3) - 504\gamma^{2}\zeta(5) - 105\gamma\zeta(2)^{3} - 630\gamma\zeta(2)\zeta(4) - 280\gamma\zeta(3)^{2}$ \\ \vspace{1mm}
$- 840\gamma\zeta(6) - 210\zeta(2)^{2}\zeta(3) - 504\zeta(2)\zeta(5) - 420\zeta(3)\zeta(4) - 720\zeta(7)$, \\ \vspace{1mm}
$\Gamma^{(8)}(1)$ \\ \vspace{1mm}
$= \gamma^{8} + 28\gamma^{6}\zeta(2) + 112\gamma^{5}\zeta(3) + 210\gamma^{4}\zeta(2)^{2} + 420\gamma^{4}\zeta(4) + 1120\gamma^{3}\zeta(2)\zeta(3)$ \\ \vspace{1mm}
$+ 1344\gamma^{3}\zeta(5) + 420\gamma^{2}\zeta(2)^{3} + 2520\gamma^{2}\zeta(2)\zeta(4) + 1120\gamma^{2}\zeta(3)^{2} + 3360\gamma^{2}\zeta(6)$ \\ \vspace{1mm}
$+ 1680\gamma\zeta(2)^{2}\zeta(3) + 4032\gamma\zeta(2)\zeta(5) + 3360\gamma\zeta(3)\zeta(4) + 5760\gamma\zeta(7) + 105\zeta(2)^{4}$ \\ \vspace{1mm}
$+ 1260\zeta(2)^{2}\zeta(4) + 1120\zeta(2)\zeta(3)^{2} + 3360\zeta(2)\zeta(6) + 2688\zeta(3)\zeta(5)$ \\ \vspace{1mm}
$+ 1260\zeta(4)^{2} + 5040\zeta(8)$, \\ \vspace{1mm}
$\Gamma^{(9)}(1)$ \\ \vspace{1mm}
$= -\gamma^{9} - 36\gamma^{7}\zeta(2) - 168\gamma^{6}\zeta(3) - 378\gamma^{5}\zeta(2)^{2} - 756\gamma^{5}\zeta(4)$ \\ \vspace{1mm}
$ - 2520\gamma^{4}\zeta(2)\zeta(3) - 3024\gamma^{4}\zeta(5) - 1260\gamma^{3}\zeta(2)^{3} - 7560\gamma^{3}\zeta(2)\zeta(4) - 3360\gamma^{3}\zeta(3)^{2}$ \\ \vspace{1mm}
$ - 10080\gamma^{3}\zeta(6) - 7560\gamma^{2}\zeta(2)^{2}\zeta(3) - 18144\gamma^{2}\zeta(2)\zeta(5) - 15120\gamma^{2}\zeta(3)\zeta(4)$ \\ \vspace{1mm}
$- 25920\gamma^{2}\zeta(7) - 945\gamma\zeta(2)^{4} - 11340\gamma\zeta(2)^{2}\zeta(4) - 10080\gamma\zeta(2)\zeta(3)^{2}$ \\ \vspace{1mm}
$- 30240\gamma\zeta(2)\zeta(6) - 24192\gamma\zeta(3)\zeta(5) - 11340\gamma\zeta(4)^{2} - 45360\gamma\zeta(8)$ \\ \vspace{1mm}
$- 2520\zeta(2)^{3}\zeta(3) - 9072\zeta(2)^{2}\zeta(5) - 15120\zeta(2)\zeta(3)\zeta(4) - 25920\zeta(2)\zeta(7)$ \\ \vspace{1mm}
$- 2240\zeta(3)^{3} - 20160\zeta(3)\zeta(6) - 18144\zeta(4)\zeta(5) - 40320\zeta(9)$, \\ \vspace{1mm}
$\Gamma^{(10)}(1)$ \\ \vspace{1mm}
$= \gamma^{10} + 45\gamma^{8}\zeta(2) + 240\gamma^{7}\zeta(3) + 630\gamma^{6}\zeta(2)^{2} + 1260\gamma^{6}\zeta(4) + 5040\gamma^{5}\zeta(2)\zeta(3)$ \\ \vspace{1mm}
$ + 6048\gamma^{5}\zeta(5) + 3150\gamma^{4}\zeta(2)^{3} + 18900\gamma^{4}\zeta(2)\zeta(4) + 8400\gamma^{4}\zeta(3)^{2} + 25200\gamma^{4}\zeta(6)$ \\ \vspace{1mm}
$ + 25200\gamma^{3}\zeta(2)^{2}\zeta(3) + 60480\gamma^{3}\zeta(2)\zeta(5) + 50400\gamma^{3}\zeta(3)\zeta(4) + 86400\gamma^{3}\zeta(7)$ \\ \vspace{1mm}
$+ 4725\gamma^{2}\zeta(2)^{4} + 56700\gamma^{2}\zeta(2)^{2}\zeta(4) + 50400\gamma^{2}\zeta(2)\zeta(3)^{2} + 151200\gamma^{2}\zeta(2)\zeta(6)$ \\ \vspace{1mm}
$+ 120960\gamma^{2}\zeta(3)\zeta(5) + 56700\gamma^{2}\zeta(4)^{2}+ 226800\gamma^{2}\zeta(8) + 25200\gamma\zeta(2)^{3}\zeta(3)$ \\ \vspace{1mm}
$+ 90720\gamma\zeta(2)^{2}\zeta(5) + 151200\gamma\zeta(2)\zeta(3)\zeta(4) + 259200\gamma\zeta(2)\zeta(7) + 22400\gamma\zeta(3)^{3}$ \\ \vspace{1mm}
$+ 201600\gamma\zeta(3)\zeta(6) + 181440\gamma\zeta(4)\zeta(5) + 403200\gamma\zeta(9) + 945\zeta(2)^{5} + 18900\zeta(2)^{3}\zeta(4)$ \\ \vspace{1mm}
$+ 25200\zeta(2)^{2}\zeta(3)^{2} + 75600\zeta(2)^{2}\zeta(6) + 120960\zeta(2)\zeta(3)\zeta(5) + 56700\zeta(2)\zeta(4)^{2}$ \\ \vspace{1mm}
$+ 226800\zeta(2)\zeta(8) + 50400\zeta(3)^{2}\zeta(4) + 172800\zeta(3)\zeta(7) + 151200\zeta(4)\zeta(6)$ \\ \vspace{1mm}
$+ 72576\zeta(5)^{2}+ 362880\zeta(10)$. \vspace{1mm}
\end{longtable}
\noindent
{\bf Acknowledgements}
The authors would like to thank Professor Shinya Koyama and Professor  Nobushige Kurokawa for useful advice.
They also express their gratitude to the referee for suggesting the addition of the generalizetion of Kummer's formula in \S \ref{sec:Kummer}.


\begin{thebibliography}{10}
\bibitem[1]{A1} T.~M.~Apostol, Introduction to Analytic Number Theory, Undergraduate Texts in Mathematics, Springer, 1976.
\bibitem[2]{A2} T.~M.~Apostol, Formulas for higher derivatives of the Riemann zeta function, Math.~Comp.~{\bf 44}, no.~169 (1985) , 223--232.
\bibitem[3]{A3} T.~M.~Apostol, Zeta and Related Functions, NIST Handbook of Mathematical Functions, U.~S. Dept. Commerce, Washington, DC
 (2010), 601--616.
\bibitem[3]{B1} E.~W.~Barnes, The theory of the double gamma function, Phil.~ Trans. Roy. Soc. (A) {\bf 196} (1901), 265--387.
\bibitem[4]{B2} E.~W.~Barnes, On the theory of the multiple gamma function, Trans. Camb. Philos. Soc. {\bf 19} (1904), 374--425.
\bibitem[5]{B} B.~C.~ Berndt, On the Hurwitz zeta function, Rochy Mountain J.~Math.~{\bf 2} (1972), 151-157.
\bibitem[6]{Co}  M.~W.~Coffey,  A set of identities for a class of alternating binomial sums arising in computing
applications, Utilitas Mathematica {\bf 76} (2008), 79--90.
\bibitem[7]{Com} L.~Comtet, Advanced combinatorics, The art of finite and infinite expansions, Revised and enlarged edition. D. Reidel Publishing Co., Dordrecht, 1974.
\bibitem[8]{KK} S.~Koyama and N.~Kurokawa, Kummer's formula for multiple gamma functions, J. Ramanujan Math. Soc. {\bf 18} (2003),   87--107.
\bibitem[9]{K} E.~Kummer, Beitrag zur Theorie der Function $\Gamma(x) = \int_{0}^{\infty} e^{-v} v^{x-1} dv$, J. reine angew. Math. {\bf 35} (1847), 1--4.
\bibitem[10]{L} M.~Lerch, Dal\v{s}i studie v oboru Malmst\'{e}novsk\'{y}ch \'{r}ad, Rozpravy \v{C}esk\'{e} Akad. v\'{e}d  {\bf 3} (1894), 1--61.
\end{thebibliography}
\end{document}